\documentclass[11pt]{article}
\usepackage[utf8]{inputenc}
\usepackage[english]{babel}
\usepackage{authblk}
\usepackage{graphicx}        
\usepackage{subfigure}                            
\usepackage{multicol}        
\usepackage[bottom]{footmisc}
\usepackage{placeins}
\usepackage{changes,todonotes}
\usepackage{pgf}
\usepackage{psfrag}
\usepackage{pgfplots}
\pgfplotsset{compat=1.18}

\usepackage{hyperref}
\usepackage{appendix}
\usepackage{amsmath,amssymb,amsfonts}
\usepackage{amsthm}
\usepackage{bm}
\usepackage{mathtools}
\usepackage{color}
\usepackage{stmaryrd}
\usepackage{blindtext}
\usepackage{caption}
\usepackage{multirow}
\usepackage{float}
\usepackage[margin=2cm]{geometry}
\usepackage{mathrsfs}%

\usepackage{xcolor}%
\usepackage{textcomp}%
\usepackage{manyfoot}%
\usepackage{booktabs}%
\usepackage{tikz}

\newcommand{\uu}{\bm{u}}
\newcommand{\vvv}{\bm{v}}

\newcommand{\ww}{\bm{w}}
\newcommand{\zz}{\bm{z}}
\newcommand{\FF}{\bm{f}}
\newcommand{\ee}{\bm{e}}
\renewcommand{\aa}{\bm{a}}

\newcommand{\ff}{\bm{f}}

\newcommand{\UU}{\bm{U}}

\newcommand{\GG}{\bm{G}}
\newcommand{\LL}{\bm{L}}
\newcommand{\VV}{\bm{V}}
\newcommand{\HH}{\bm{H}}

\newcommand{\calA}{{\mathcal A}}

\newcommand{\calB}{{\mathcal B}}

\newcommand{\vertiii}[1]{{\left\vert\kern-0.25ex\left\vert\kern-0.25ex\left\vert #1 
    \right\vert\kern-0.25ex\right\vert\kern-0.25ex\right\vert}}

\newtheorem{myth}{Theorem}
\newtheorem{myprop}{Proposition}

\newtheorem{mylemma}{Lemma}

\newtheorem{assumption}{Assumption}
\newtheorem{remark}{Remark}

\newcommand{\trinorm}[1]{{\vert\kern-0.25ex\vert\kern-0.25ex\vert #1 \vert\kern-0.25ex\vert\kern-0.25ex\vert}}

\makeindex             

\begin{document}

\title{A review of discontinuous Galerkin time-stepping methods for wave propagation problems}

\author[$\star$]{Paola F. Antonietti}
\author[$\star$,$\diamond$]{Alberto Artoni}
\author[$\star$]{Gabriele Ciaramella}
\author[$\star$]{Ilario Mazzieri}

\affil[$\star$]{MOX-Laboratory for Modeling and Scientific Computing, Department of Mathematics, Politecnico di Milano, P.za L. da Vinci 32, 20133, Milan, Italy.}

\affil[ ]{\texttt{\{paola.antonietti, gabriele.ciaramella, ilario.mazzieri\}@polimi.it}}

\affil[$\diamond$]{Leonardo Innovation Labs \& IP, 
Via Pieragostini, 80, 16151, Genova, Italy.}

\affil[ ]{\texttt{alberto.artoni1995@gmail.com}}

\maketitle


\abstract{This chapter reviews and compares discontinuous Galerkin time-stepping methods for the numerical approximation of second-order ordinary differential equations, particularly those stemming from space finite element discretization of wave propagation problems. Two formulations, tailored for second- and first-order systems of ordinary differential equations, are discussed within a generalized framework, assessing their stability, accuracy, and computational efficiency. Theoretical results are supported by various illustrative examples that validate the findings, enhancing the understanding and applicability of these methods in practical scenarios.}

\section{Introduction}

This chapter presents a detailed review and comparison of high-order discontinuous Galerkin (dG) finite element methods introduced in \cite{DalSanto2018} and \cite{AntoniettiMiglioriniMazzieri2021}, designed for the temporal integration of second-order differential problems, such as those encountered in wave propagation phenomena in acoustic, elastic, and porous media. Time integration of such problems often involves implicit and explicit finite differences, Leap-frog, Runge-Kutta, or Newmark schemes, as extensively reviewed in the literature \cite{Ve07,Bu08,QuSaSa07}. Generally, explicit time integration schemes are favored over implicit due to their computational efficiency and straightforward implementation. However, while implicit methods are unconditionally stable, they tend to be computationally expensive. Conversely, explicit methods, though computationally cheaper, are conditionally stable. The time step choice dictated by the Courant-Friedrichs-Lewy (CFL) condition can occasionally present limitations. Various strategies have been proposed to address this limitation, including local time-stepping algorithms \cite{DiGr09,GrMi13,Hochbruck2016} and the tent-pitching scheme \cite{GoScWi17}. These algorithms impose the CFL condition element-wise, allowing for an optimal choice of the time step, but they necessitate additional synchronization processes to ensure proper wavefield propagation between elements.\\

This work is concerned with implicit time integration methods based on a dG paradigm. Originally developed for the space approximation of hyperbolic problems \cite{ReedHill73,Lesaint74}, dG methods have been extended to space discretization of elliptic and parabolic equations \cite{wheeler1978elliptic,arnold1982interior,HoScSu00,CockKarnShu00, riviere2008discontinuous,HestWar,DiPiEr}. 
Concerning initial-value problems, dG methods have shown advantages over other implicit schemes, such as Johnson's method \cite{JOHNSON1993,ADJERID2011} thanks to the fact that each time slab of the solution depends only on the previous time step. Moreover, these methods readily integrate with dG space discretization techniques, enabling the discretization of both space and time dimensions through the dG paradigm. These methods offer high-order convergence of errors through refinement and can achieve stability with local CFL conditions \cite{MoRi05}. They can be categorized into structured and unstructured techniques. Structured techniques \cite{Tezduyar06} utilize a Cartesian product grid of spatial mesh and time partition, while unstructured techniques \cite{Hughes88,Idesman07} treat time as an additional dimension. Among unstructured methods, Trefftz techniques \cite{2016_KretzschmarMoiolaPerugiaSchnepp,BaGeLi17,BaCaDiSh18} and tent-pitching paradigms \cite{GoScWi17} have been proposed to improve stability. 
Space-time discontinuous Galerkin approximation of wave propagation problems have been proposed in  \cite{2009_PetersenFarhatTezaur,2021_BansalMoiolaPerugiaSchwab,2022_ShuklaVanderVegt}.
Space–time Trefftz dG methods for the wave equation have been studied in \cite{2016_KretzschmarMoiolaPerugiaSchnepp, 2018_MoiolaPerugia,2023_ImbertMoiolaStocker}. In \cite{2020_PerugiaSchoberlStockerWintersteiger}
tent pitching and Trefftz-DG method for the wave equation have been addressed. Space-time dG methods for parabolic problems on prismatic meshes have been proposed in \cite{2017_CangianiDongGeorgoulis}, whereas 
dG time stepping methods for semilinear parabolic problems have been addressed in \cite{2020_CangianiGeorgoulisSabawi}. Recently, space-time methods for Virtual Element discretization of parabolic equations have been developed in \cite{2024_GomezMascottoMoiolaPerugia,2024_GomezMasscottoPerugia}.\\

In this work we provide a comprehensive review and comparison of the dG time integration approach presented in \cite{DalSanto2018} for second-order systems and in \cite{AntoniettiMiglioriniMazzieri2021} for first-order systems of ordinary differential equations stemming from space discretization of wave propagation problems. Both methods are presented in a general framework and reviewed from the point of view of stability, accuracy, and computational efficiency. 
We also refer the reader to \cite{CiaramellaGanderMazzieri2025} for equivalence proofs between the proposed dG schemes and classical time-stepping methods
(such as Newmark schemes, general linear methods, and Runge-Kutta methods) and a thorough analysis in terms of accuracy, consistency, and computational cost.

The chapter is organized as follows. In Section~\ref{sec:model_problem} we formulate the problem and introduce the finite element setting employed for the dG formulations considered in Section~\ref{sec::dg_form_I} and \ref{sec::dg_form_II}. The stability and convergence properties are discussed in Section~\ref{sec:Convergence}, where we review in detail the \textit{a priori} error estimates. In Section~\ref{sec::AlgebraicFormulation}, we rewrite the equations into the corresponding algebraic systems, and we discuss suitable solution strategies. Finally, in Section~\ref{sec::numerical_results}, we compare the considered methods through several numerical experiments, including wave propagation in acoustic, poroelastic, and coupled acoustic-poroelastic media.\\

Throughout the chapter, we denote by $||\aa||$ the Euclidean norm of a vector $\aa \in \mathbb{R}^d$, $d\ge 1$ and by $||A||_{\infty} = \max_{i=1,\dots,m}\sum_{j=1}^n |a_{ij}|$, the $\ell^{\infty}$-norm of a matrix $A\in\mathbb{R}^{m\times n}$, $m,n\ge1$. For a given $I\subset\mathbb{R}$ and $v: I\rightarrow\mathbb{R}$ we denote by $L^p(I)$ and $H^p(I)$, $p\in\mathbb{N}_0$, the classical Lebesgue and Hilbert spaces, respectively, and endow them with the usual norms, see \cite{AdamsFournier2003}. Finally, we indicate the Lebesgue and Hilbert spaces for vector-valued functions as $\LL^p(I) = [L^p(I)]^d$ and $\HH^p(I) = [H^p(I)]^d$, $d\ge1$, respectively.

\section{Model problem and finite element settings}\label{sec:model_problem}

For $T>0$, we consider the following model problem \cite{kroopnick}: find $\uu(t) \in\HH^2(0,T]$ such that 
\begin{equation}
	\label{Eq:SecondOrderEquation}
	\begin{cases}
		M\ddot{\uu}(t) + D\dot{\uu}(t) + A\uu(t) = \ff(t) \qquad \forall\, t \in (0,T], \\
		(\uu,\dot{\uu})(0) = (\uu_0,\vvv_0), 
	\end{cases}
\end{equation}
where $M,D, A \in \mathbb{R}^{{N_h\times N_h}}$ are symmetric matrices,  $\uu_0, \vvv_0 \in \mathbb{R}^{{N_h}}$ and $\ff \in \LL^2(0,T]$.
Problem~\eqref{Eq:SecondOrderEquation} can be recast in the  first order form by introducing a variable $\vvv \in\HH^1(0,T]$ as follows:
\begin{equation}
\label{Eq:FirstOrderSystem1}
	\begin{cases}
		\dot{\uu}(t) - \vvv(t) = \boldsymbol{0} &\forall\, t\in(0,T], \\
		M\dot{\vvv}(t) + D\vvv(t) + A\uu(t) = \ff(t) &\forall\, t\in(0,T], \\
		(\uu,\vvv)(0) = (\uu_0,\vvv_0).
	\end{cases}
\end{equation}
We are mainly interested in the following cases, which we present in the form of assumption on the matrices $M,A,D$.
\begin{assumption}\label{ass:matrices}

We assume that the symmetric matrices $M,D, A \in \mathbb{R}^{{N_h\times N_h}}$ satisfy one of the following
\begin{itemize}
    \item[a)] $M,A$ are positive definite and $D=0$ (cf. acoustic wave propagation in Section~\ref{sec::acoustic_example}),
    \item[b)] $M,D$ are positive definite and $A$ is positive semidefinite (poroelastic and poroelastic-acoustic wave propagation, cf. Section~\ref{sec::poro_example}). 
\end{itemize}
\end{assumption}

To integrate in time systems \eqref{Eq:SecondOrderEquation} or \eqref{Eq:FirstOrderSystem1}, we introduce a partition of the interval $I=(0,T]$ into $N$ time-intervals $I_n = (t_{n-1},t_n]$ of length $\Delta t_n = t_n-t_{n-1}$, for $n=1,\dots,N$ with $t_0 = 0$ and $t_N = T$, cf. Figure~\ref{Fig:TimeDomain}.
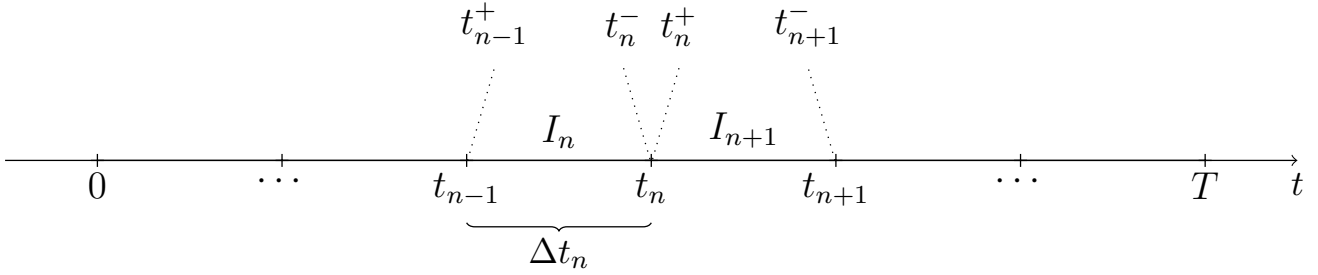
\begin{figure}[!htbp]
	\centering
	{\resizebox{1\textwidth}{!}{\definecolor{mycolor1}{rgb}{0.00000,0.49804,0.00000}%
\begin{tikzpicture}

\draw [->]  (-1,0) -- (13,0) node [below] {$t$};
\draw plot [mark=+, smooth, line width=7pt] coordinates
{(0,0) (2,0) (4,0) (6,0) (8,0) (10,0) (12,0)  };

\node [font=\large] at (0,0) [below] {0};
\node [font=\large] at (2,0) [below] {$\cdots$};
\node [font=\large] at (4,0) [below] {$t_{n-1}$};
\node [font=\large] at (5,0) [above] {$I_{n}$};
\node [font=\large] at (6,0) [below] {$t_{n}$};
\node [font=\large] at (7,0) [above] {$I_{n+1}$};
\node [font=\large] at (8,0) [below] {$t_{n+1}$};
\node [font=\large] at (10,0) [below] {$\cdots$};
\node [font=\large] at (12,0) [below] {$T$};


\draw [dotted] (5.7,1) -- (5.99,0);
\draw [dotted] (6.3,1) -- (6.01,0);

\draw [dotted] (4.01,0) -- (4.3,1);
\draw [dotted] (7.99,0) -- (7.7,1);

\node [font=\large] at (6.3,1.1) [above] {$t_{n}^{+}$};
\node [font=\large] at (5.7,1.1) [above] {$t_n^-$};

\node [font=\large] at (4.3,1.1) [above] {$t_{n-1}^+$};
\node [font=\large] at (7.7,1.1) [above] {$t_{n+1}^-$};

\draw[decoration={brace,mirror,raise=5pt},decorate]
  (4,-0.5) -- node[below=6pt] {$\Delta t_n$} (6,-0.5);

\end{tikzpicture}}}
	\caption{Time domain partition: values $t_n^+$ and $t_n^-$ are also reported. }\label{Fig:TimeDomain}
\end{figure}
Next, we incrementally build (on $n$) an approximation of the exact solution $\uu$
(and, for system \eqref{Eq:FirstOrderSystem1} also for $\vvv$) in each time slab $I_n$.  To denote the scalar product between two vectors $\ww, \zz \in \mathbb{R}^{{N_h}}$ we use the following notation
\begin{equation*}
	(\ww,\zz)_I = \int_I \ww(s)\cdot\zz(s)\text{d}s, 
\end{equation*}
and denote by $[\ww]_n$, the jump operator of a (sufficiently regular) function $\ww$ at $t_n$ as
\begin{equation*}
	[\ww]_n = \ww(t_n^+) - \ww(t_n^-) = \ww^+ -\ww^-, \quad \text{for } n\ge 0,
\end{equation*}
where
\begin{equation*}
	\ww(t_n^\pm) = \lim_{\epsilon\rightarrow 0^\pm}\ww(t_n+\epsilon), \quad \text{for } n\ge 0.
\end{equation*}
Next, we introduce the functional spaces
\begin{equation*}
	\label{Eq:PolynomialSpace}
	V_n^{r_n} = \{ \ww:I_n\rightarrow\mathbb{R}^{N_h} \text{ s.t. }  \ww\in[\mathcal{P}^{r_n}(I_n)]^{N_h} \},
\end{equation*}
where $\mathcal{P}^{r_n}(I_n)$ is the space of polynomial defined on $I_n$ of maximum degree $r_n \geq 1$, and 
\begin{equation*}
	\label{Eq:DGSpace}
	\mathcal{V}_{dG} = \{ \ww\in\LL^2(0,T] \text{ s.t. } \ww|_{I_n} \in V_n^{r_n} \}.
\end{equation*}

\textcolor{black}{In the following sections, we will first introduce \textit{dG2}, namely the second order formulation of problem \eqref{Eq:SecondOrderEquation}. Then, we will present \textit{dG1}, namely the first-order formulation of \eqref{Eq:FirstOrderSystem1}}.

\subsection{Formulation \textit{dG2}}\label{sec::dg_form_I}
To obtain the discontinuous formulation of problem \eqref{Eq:SecondOrderEquation} we focus on the generic interval $I_n$,  we multiply the first equation in \eqref{Eq:SecondOrderEquation} by a (regular enough) test function $\dot{\ww}(t)$, integrate in time over $I_n$, and add the null terms $M[\dot{\uu}]_{n-1}\cdot\dot{\ww}(t_{n-1}^+)$ and $A[\uu]_{n-1}\cdot\ww(t_{n-1}^+)$, since $\bm u \in \bm H^2(0,T)$,  to get 
\begin{equation}
	\label{Eq:Weak1_dgI_slab}
	(M\ddot{\uu},\dot{\ww})_{I_n} + (D\dot{\uu},\dot{\ww})_{I_n} + (A\uu,\dot{\ww})_{I_n}  +  M[\dot{\uu}]_{n-1}\cdot\dot{\ww}(t_{n-1}^+) + A[\uu]_{n-1}\cdot\ww(t_{n-1}^+) = (\FF,\dot{\ww})_{I_n}.
\end{equation}
By summing over all intervals $I_n$, for $n=1,\ldots,N$ we obtain the problem: \\ \textcolor{black}{(\textit{dG2})} find $\uu_{dG} \in \mathcal{V}_{dG}$  such that
\begin{equation}\label{eq::wf_dgI}
    \mathcal{A}(\uu_{dG},\ww) = \mathcal{F}(\ww) \quad \forall \; \ww \in \mathcal{V}_{dG},
\end{equation}
where for any $\uu,\ww \in \mathcal{V}_{dG}$ it holds
\begin{align}\label{def::A_dgI}
    \mathcal{A}(\uu,\ww) & =  \sum_{n=1}^{N}(M\ddot{\uu},\dot{\ww})_{I_n} + (D\dot{\uu},\dot{\ww})_{I_n} + (A\uu,\dot{\ww})_{I_n} \nonumber\\ & + \sum_{n=1}^{N-1} M[\dot{\uu}]_{n}\cdot\dot{\ww}(t_{n}^+) + A[\uu]_{n}\cdot\ww(t_{n}^+) + M\dot{\uu}(0^+)\cdot\dot{\ww}(0^+) + A \uu(0^+)\cdot\ww(0^+), 
\end{align}
and for any $\ww \in \mathcal{V}_{dG}$ we have
\begin{align}\label{def::F_dgI}
    \mathcal{F}(\ww) & =  \sum_{n=1}^{N}(\FF,\dot{\ww})_{I_n} + M\vvv_0\cdot\dot{\ww}(0^+) + A \uu_0\cdot\ww(0^+).
\end{align}
Under the previous general assumptions on the matrices $A,M$ and $D$, the bilinear form $\mathcal{A}(\cdot,\cdot)$ induces a seminorm on the space $\mathcal{V}_{dG}$, which we denote by $| \cdot |_{\calA}$ and is defined as
\begin{align}\label{def::normA}
    | \uu |_{\calA}^2 = \calA(\uu,\uu)  =& \sum_{n=1}^N ||D\dot{\uu}||_{0,I_n}^2 + \frac{1}{2}(M^{\frac{1}{2}}\dot{\uu}(0^+))^2 + \frac{1}{2}\sum_{n=1}^{N-1}(M^{\frac{1}{2}}[\dot{\uu}]_n)^2 + \frac{1}{2}(M^{\frac{1}{2}}\dot{\uu}(T^-))^2 \nonumber \\ & + \frac{1}{2}(A^{\frac{1}{2}}{\uu}(0^+))^2 + \frac{1}{2}\sum_{n=1}^{N-1}(A^{\frac{1}{2}}[{\uu}]_n)^2 + \frac{1}{2}(A^{\frac{1}{2}}{\uu}(T^-))^2.
\end{align}
{We remark that \eqref{def::normA} is} a norm when $A,M$ and $D$ are symmetric and positive definite.

We next state the following well-posedness result for Problem~\eqref{eq::wf_dgI}. 
\begin{myth}\label{teo::wellpos_dgI}
Let Assumption~\ref{ass:matrices} holds. Then, Problem~\eqref{eq::wf_dgI} admits a unique solution.
\end{myth}
\begin{proof}
 Due to the linearity of the problem \eqref{eq::wf_dgI}, it is enough to consider the data $\FF = \vvv_0 = \uu =  \bm 0$ so that $\mathcal{F}(\ww) = 0$ for all $\ww \in \mathcal{V}_{dG}$, and prove that $\uu_{dG} = \bm 0$.
By choosing $\ww = \uu_{dG}$, from \eqref{def::normA} we obtain $| \uu_{dG}|_{\calA} = 0$, i.e. all jumps and traces 
present in \eqref{def::normA} are null, namely $\dot{\uu}(0^+)=[\dot{\uu}]_n = \dot{\uu}(T^-) = A^{\frac{1}{2}}{\uu}(0^+) = A^{\frac{1}{2}}[{\uu}]_n = A^{\frac{1}{2}}{\uu}(T^-) = \bm 0 $. Then, \eqref{eq::wf_dgI} becomes
\begin{equation*}
    \sum_{n=1}^{N}(M\ddot{\uu}_{dG} + D\dot{\uu}_{dG} + A\uu_{dG},\dot{\ww})_{I_n}  = \bm  0.
\end{equation*}
Choosing $\dot{\ww} = M\ddot{\uu}_{dG} + D\dot{\uu}_{dG} + A\uu_{dG}$ we can deduce that $\uu_{dG}$ is solution  of the differential problem  
\begin{equation}\label{eq:diff_prob_I}
M\ddot{\uu}_{dG} + D\dot{\uu}_{dG} + A\uu_{dG}  = 0    
\end{equation}
for any $I_n$. For the first time-slab, $\uu_{dG}$ is a polynomial solution (or an analytical solution) of \eqref{eq:diff_prob_I} with homogeneous initial conditions, so it vanishes in $I_1$. Iterating over the time slabs and using that the jumps at time instants $t_n$ vanish, we deduce that $\uu_{dG}$ is zero in the whole time interval $[0,T)$. Existence of the solution follows from linearity and finite dimensionality.
\end{proof}

\subsection{Formulation \textit{dG1}}\label{sec::dg_form_II}
Proceeding similarly as in Section~\ref{sec::dg_form_I}, we can obtain the weak formulation for problem \eqref{Eq:FirstOrderSystem1}. Indeed, focusing on $I_n$,  we multiply the first (resp. second) equation in \eqref{Eq:FirstOrderSystem1} by a (regular enough) test function $\ww(t)$ (resp. $\zz(t)$), integrate in time over $I_n$ and add the null terms $A[\uu]_{n-1}\cdot\ww(t_{n-1}^+)$ and $M[\vvv]_{n-1}\cdot\zz(t_{n-1}^+)$ as follows  
\begin{equation}\label{Eq:Weak1_dgII_slab}
\begin{cases}
(\dot{\uu},\ww)_{I_n} - (\vvv,\ww)_{I_n} + [\uu]_{n-1}\cdot\ww(t_{n-1}^+) = 0, & \\
(M\dot{\vvv},\zz)_{I_n} + (D\vvv,\zz)_{I_n} + (A\uu,\zz)_{I_n}  +  M[\vvv]_{n-1}\cdot\zz(t_{n-1}^+) = (\FF,\zz)_{I_n}, & \\ 
\end{cases}
\end{equation}
since $\bm u, \bm v \in H^1(0,T)$.
Summing up over all time intervals we obtain the problem: \\ \textcolor{black}{(\textit{dG1})} find $(\uu_{dG}, \vvv_{dG}) \in \mathcal{V}_{dG} \times \mathcal{V}_{dG}$  such that
\begin{equation}\label{eq::wf_dgII}
    \mathcal{B}((\uu_{dG},\vvv_{dG}),(\ww,\zz)) = \mathcal{G}((\ww,\zz)) \quad \forall \; (\ww,\zz) \in \mathcal{V}_{dG}\times \mathcal{V}_{dG},
\end{equation}
where for any $(\uu,\vvv),(\ww,\zz) \in \mathcal{V}_{dG}\times \mathcal{V}_{dG}$ it holds
\begin{align}\label{def::A_dgII}
    \mathcal{B}((\uu,\vvv),(\ww,\zz)) & =  \sum_{n=1}^{N}(\dot{\uu},\ww)_{I_n} - (\vvv,\ww)_{I_n} + \sum_{n=1}^{N-1} [\uu]_{n}\cdot\ww(t_{n}^+) +  \uu(0^+)\cdot\ww(0^+)\nonumber\\ & + \sum_{n=1}^{N}(M\dot{\vvv},\zz)_{I_n} + (D\vvv,\zz)_{I_n} + (A\uu,\zz)_{I_n}   + \sum_{n=1}^{N-1}   M[\vvv]_{n}\cdot\zz(t_{n}^+) \nonumber \\& + M\vvv(0^+)\cdot\zz(0^+) , 
\end{align}
and for any $(\ww,\zz) \in \mathcal{V}_{dG}\times \mathcal{V}_{dG}$ we have
\begin{align}\label{def::F_dgII}
    \mathcal{G}((\ww,\zz)) & =  \sum_{n=1}^{N}(\FF,\zz)_{I_n} + M\vvv_0\cdot\zz(0^+) +  \uu_0\cdot\ww(0^+). 
\end{align}
{Under Assumption~\ref{ass:matrices} on the matrices $A,M$ and $D$,}  the bilinear form $\mathcal{B}(\cdot,\cdot)$ induces a seminorm on the space $\mathcal{V}_{dG}\times\mathcal{V}_{dG}$ which we denote by $| \cdot |_{\calB}$ and it is defined as
\begin{align}\label{def::normB}
    | (\uu,\vvv) |_{\calB}^2 & = \calB((\uu,\vvv),(A\uu,\vvv)) \nonumber \\ & = \sum_{n=1}^N ||D\vvv||_{0,I_n}^2 + \frac{1}{2}(M^{\frac{1}{2}}\vvv(0^+))^2 + \frac{1}{2}\sum_{n=1}^{N-1}(M^{\frac{1}{2}}[\vvv]_n)^2 + \frac{1}{2}(M^{\frac{1}{2}}\vvv(T^-))^2 \nonumber \\ & + \frac{1}{2}(A^{\frac{1}{2}}{\uu}(0^+))^2 + \frac{1}{2}\sum_{n=1}^{N-1}(A^{\frac{1}{2}}[{\uu}]_n)^2 + \frac{1}{2}(A^{\frac{1}{2}}{\uu}(T^-))^2.
\end{align}
Notice that, in contrast to the case $| \cdot |^2_{\mathcal{A}}$, the positivity of the matrices $A,M$ and $D$ does not imply $| \cdot |^2_{\mathcal{B}}$ to be a norm.
\begin{remark}
    In the literature, Formulation~\eqref{eq::wf_dgII} is typically presented in the following form
    \begin{equation*}
    \begin{cases}
(A\dot{\uu},\ww)_{I_n} - (A\vvv,\ww)_{I_n} + A[\uu]_{n-1}\cdot\ww(t_{n-1}^+) = 0, & \\
(M\dot{\vvv},\zz)_{I_n} + (D\vvv,\zz)_{I_n} + (A\uu,\zz)_{I_n}  +  M[\vvv]_{n-1}\cdot\zz(t_{n-1}^+) = (\FF,\zz)_{I_n}, & \\ 
\end{cases}
\end{equation*}
see, e.g., \cite{HULBERT1990,AntoniettiMiglioriniMazzieri2021,rezaei2023discontinuous}, in which the first equation is multiplied by the matrix $A$. This is equivalent to \eqref{eq::wf_dgII} only if $A$ is positive definite. This is not verified for instance in the Biot model of poroelasticity, see Section~\ref{sec::poro_example}.
\end{remark}

Next, we recall the following result, for which the proof closely follows the arguments of the proof of Theorem~\ref{teo::wellpos_dgI}, which we omit for the sake of presentation. We refer to  \cite{AntoniettiMiglioriniMazzieri2021} for more details.
\begin{myth}[Well posedness of \eqref{eq::wf_dgII}]
Let Assumption~\ref{ass:matrices} holds.  Then,
Problem~\eqref{eq::wf_dgII} admits a unique solution.
\end{myth}

\section{Stability and convergence analysis}\label{sec:Convergence}

In this section, we report for both formulations in Section~\ref{sec::dg_form_I} and Section~\ref{sec::dg_form_II} the main stability results and use the latter to prove  error estimate for the numerical error, in the energy seminorms \eqref{def::normA} and  \eqref{def::normB}, respectively.
 
\subsection{Formulation \textit{dG2}}

In the following, we report the main stability and convergence results for the \textit{dG2} formulation. 
We refer the reader to \cite{DalSanto2018,AnMaMi20} for the details.

\begin{myprop}[\textit{Stability of dG2}]
{Let Assumption~\ref{ass:matrices} holds.}	Let $\ff \in \LL^2(0,T]$, $\uu_0, \vvv_0 \in \mathbb{R}^{{N_h}}$, and let $\uu_{dG} \in \mathcal{V}_{dG}$ be the solution of \eqref{eq::wf_dgI}, then it holds
	\begin{equation}
		\label{eq::stability_dgI}
		|\uu_{dG}|_{\calA}^2 \lesssim \sum_{n=1}^N ||D^{-\frac{1}{2}}\ff||_{0,I_n}^2+(A^{\frac{1}{2}}\uu_0)^2+(M^{\frac{1}{2}}\vvv_0)^2.
	\end{equation}
\end{myprop}
\begin{proof}
See \cite{DalSanto2018,AnMaMi20}.
\end{proof}

\begin{myth}[\textit{Convergence of dG2}]\label{thm::conv_dgI}
Let Assumption~\ref{ass:matrices} holds. Let $\uu$ be the solution of \eqref{Eq:SecondOrderEquation} and let $\uu_{dG} \in \mathcal{V}_{dG}$ be its dG finite element approximation based on employing formulation \eqref{eq::wf_dgI}. If $\uu_{|_{I_n}} \in \bm H^{s_n} (I_n)$, for any $n = 1, \ldots, N$, with $s_n \geq 2$,
then
\begin{equation}\label{eq::error_dgI}
    | \uu - \uu_{dG} |_{\calA} \lesssim \sum_{n=1}^N \frac{\Delta t_n^{\mu_n-\frac12}}{r_n^{s_n-3}} \| \uu \|_{\bm H^{\mu_n+1}(I_n)},   
\end{equation}
where $\mu_n = min(r_n, s_n )$ for any $n = 1, \ldots, N$, and where the hidden constant depends on the norm of
the matrices $D$ and $A$.
\end{myth}

\begin{proof}
    For the proof see \cite{DalSanto2018}.
\end{proof}

\begin{remark}
 For the undamped case, i.e., $D=0$ in  Assumption~\ref{ass:matrices},  inequalities \eqref{eq::stability_dgI} and \eqref{eq::error_dgI} hold only for polynomial degree $r_n = 1,2$ for any $n=1,\ldots,N$. In particular, for linear polynomials, the proof follows easily the arguments in \cite{DalSanto2018,AnMaMi20} by observing  that \eqref{def::normA} reduces to 
\begin{align*}
    |\uu|_{\calA}^2 & =  \frac{1}{2}(M^{\frac{1}{2}}\dot{\uu}(0^+))^2 + \frac{1}{2}\sum_{n=1}^{N-1}(M^{\frac{1}{2}}(\dot{\uu}(t_{n-1}^+) - \dot{\uu}(t_{n}^+)))^2 + \frac{1}{2}(M^{\frac{1}{2}}\dot{\uu}(T^-))^2 \nonumber \\ & + \frac{1}{2}(A^{\frac{1}{2}}{\uu}(0^+))^2 + \frac{1}{2}\sum_{n=1}^{N-1}(A^{\frac{1}{2}}[{\uu}]_n)^2 + \frac{1}{2}(A^{\frac{1}{2}}{\uu}(T^-))^2,
\end{align*}
since $\dot{\uu}(t_n^+) = \dot{\uu}(t_{n+1}^-)$ for any $n=0,\ldots,N-1$.
\end{remark}
\subsection{Formulation \textit{dG1}}

In this section, we review the results presented in \cite{AntoniettiMiglioriniMazzieri2021} and extend them to cover case b) in  Assumption~\ref{ass:matrices}, which is relevant for the problems studied in \cite{ABM_Vietnam,AntoniettiMazzieriNatipoltri2021} and shown in Section~\ref{sec::numerical_results}. 
\begin{myprop}[\textit{Stability of dG1}] Let Assumption~\ref{ass:matrices} a) with $r_n=0,1$ for any $n=1,\ldots,N$, or Assumption~\ref{ass:matrices} b) with $r_n \geq 1$ for any $n=1,\ldots,N$ hold. 
	Let $\ff \in \LL^2(0,T]$, $\uu_0, \vvv_0 \in \mathbb{R}^{{N_h}}$, and let $(\uu_{dG},\vvv_{dG}) \in \mathcal{V}_{dG} \times \mathcal{V}_{dG}$ be the solution of \eqref{eq::wf_dgII}, then it holds
	\begin{equation}
		\label{eq::stability_dgII}
		|(\uu_{dG},\vvv_{dG})|_{\calB}^2 \lesssim \sum_{n=1}^N ||D^{-\frac{1}{2}}\ff||_{0,I_n}^2+(A^{\frac{1}{2}}\uu_0)^2+(M^{\frac{1}{2}}\vvv_0)^2. 
	\end{equation}
\end{myprop}
\begin{proof}
For the proof we refer the reader to  \cite{rezaei2023discontinuous}, for Assumption~\ref{ass:matrices} a), and \cite{AntoniettiMiglioriniMazzieri2021} for Assumption~\ref{ass:matrices} b). 
\end{proof}


Before presenting the the convergence results of this section, namely Theorem \ref{thm::conv_dgII}, we introduce some preliminary results that are instrumental for its proof.  We refer the interested reader to \cite{AntoniettiMiglioriniMazzieri2021} for further details, see also \cite{ScSc2000}.

\begin{mylemma}
	\label{Le:Projector}
	Let $I=(-1,1)$ and $u\in L^2(I)$ continuous at $t=1$, the projector $\Pi^r u \in \mathcal{P}^r(I)$, $r\in\mathbb{N}_0$, defined by the $r+1$ conditions
	\begin{equation}
		\label{Eq:Projector}
		\Pi^r u (1) = u(1), \qquad (\Pi^r u - u,q)_{I} = 0 \quad\forall\, q\in\mathcal{P}^{r-1}(I),
	\end{equation} 
	is well-posed. Moreover, let $I=(a,b)$, $\Delta t = b-a$, $r\in\mathbb{N}_0$ and $u\in H^{s_0+1}(I)$ for some $s_0\in\mathbb{N}_0$. Then
	\begin{equation}
		\label{Eq:ProjectionError}
		||u-\Pi^r u||_{L^2(I)}^2 \le C\bigg(\frac{\Delta t}{2}\bigg)^{2(s+1)}\frac{1}{r^2}\frac{(r-s)!}{(r+s)!}||u^{(s+1)}||_{L^2(I)}^2,
	\end{equation}
 and
 	\begin{equation}\label{Eq:DerivativeProjectionError}
	||\dot{u}-\dot{\big(\Pi^r u\big)}||_{L^2(I)}^2 \lesssim \bigg(\frac{\Delta t}{2}\bigg)^{2(s+1)}(r+2)\frac{(r-s)!}{(r+s)!}||u^{(s+1)}||_{L^2(I)}^2,
	\end{equation}
	for any integer $0\le s \le \min(r,s_0)$. C depends on $s_0$ but is independent from $r$ and $\Delta t$. 
\end{mylemma}

In the following, we make use of the notation 
$$\ee^u = \uu - \uu_{dG} \quad {\rm and} \quad \ee^v = \vvv - \vvv_{dG},$$
and we split the errors $\ee^u$ and $\ee^v$ as
\begin{align*}
  \ee^u & = \ee_\pi^u + \ee_h^u = (\uu - \bm \Pi^r \uu) + (\bm \Pi^r \uu - \uu_{dG}), \\
  \ee^v & = \ee_\pi^v + \ee_h^v = (\vvv - \bm \Pi^r \vvv) + (\bm \Pi^r \vvv - \vvv_{dG}),
\end{align*} 
being $\bm \Pi^r\uu$ the trivial vectorial extension of $\Pi^ru$ in Lemma~\ref{Le:Projector}. 
Finally, we remark that \eqref{eq::wf_dgII} is a consistent formulation, i.e.,    
\begin{equation}\label{eq::consistency_B}
       \calB((\ee^u,\ee^v),(\ww,\zz)) = 0 \quad \forall \,(\ww,\zz) \in \mathcal{V}_{dG} \times \mathcal{V}_{dG}.
\end{equation}
\begin{mylemma}
Let Assumption~\ref{ass:matrices} holds. Let $(\uu,\vvv)$ be the solution of \eqref{Eq:FirstOrderSystem1} and let $(\uu_{dG},\vvv_{dG}) \in \mathcal{V}_{dG}\times \mathcal{V}_{dG}$ be its dG finite element approximation based on employing formulation \eqref{eq::wf_dgII}. Then, it holds
\begin{equation}\label{eq::estimateAu}
 \sum_{n=1}^{N} \| A\ee_h^u \|^2_{0,I_n}
 \lesssim  | (\ee_h^u, \ee_h^v) |^2_{\calB} + 
 \sum_{n=1}^{N} \|\ee_\pi^v \|_{0,I_n}^2 +  \| \ee_\pi^u\|_{0,I_n}^2,
\end{equation}
   where the hidden constant depends on the norm of
the matrices $M,D$ and $A$. 
\end{mylemma}

\begin{proof}
We consider \eqref{eq::consistency_B} and choose $\ww = A \ee_h^v$ and $\zz = M^{-1}A\ee_h^u$  to obtain 
   \begin{equation*}
       \calB_a((\ee_h^u,\ee_h^v),(A \ee_h^v,M^{-1}A \ee_h^u)) = - \calB_a((\ee_\pi^u,\ee_\pi^v),(A \ee_h^v,M^{-1} A \ee_h^u)), 
   \end{equation*}
or, equivalently, 
\begin{multline}\label{eq::starting}
\sum_{n=1}^{N}(\dot{\ee_h^u},A \ee_h^v)_{I_n} - \sum_{n=1}^{N}(\ee_h^v,A \ee_h^v)_{I_n} + \sum_{n=0}^{N-1} [\ee_h^u]_{n}\cdot A \ee_h^v(t_{n}^+)  + \sum_{n=1}^{N}(\dot{\ee_h^v},A \ee_h^u)_{I_n} \\ + \sum_{n=1}^{N}(M^{-1}D\ee_h^v,A \ee_h^u)_{I_n}  + \sum_{n=1}^{N}(M^{-1}A\ee_h^u,A \ee_h^u)_{I_n}   + \sum_{n=0}^{N-1}   [\ee_h^v]_{n}\cdot A \ee_h^u(t_{n}^+) \\ = \sum_{n=1}^{N}- (\dot{\ee_\pi^u},A \ee_h^v)_{I_n} + \sum_{n=1}^{N}(\ee_\pi^v,A \ee_h^v)_{I_n} - \sum_{n=0}^{N-1} [\ee_\pi^u]_{n}\cdot A \ee_h^v(t_{n}^+)  - \sum_{n=1}^{N}(\dot{\ee_\pi^v},A \ee_h^u)_{I_n} \\ - \sum_{n=1}^{N}(M^{-1}D\ee_\pi^v,A \ee_h^u)_{I_n}  - \sum_{n=1}^{N}(M^{-1}A\ee_\pi^u,A \ee_h^u)_{I_n}   - \sum_{n=0}^{N-1}   [\ee_\pi^v]_{n}\cdot A \ee_h^u(t_{n}^+),
\end{multline}
where we have implicitly used that $\ee_h^u(0^-) = \ee_h^v(0^-) = \ee_\pi^u(0^-) = \ee_\pi^v(0^-) =   \bm 0$.
Next, we integrateby parts the first term on the left-hand side of \eqref{eq::starting} to get
\begin{equation}\label{eq::simplify_lhs}
    \sum_{n=1}^N (\dot{\ee}_h^u,A \ee_h^v)_{I_n} = -\sum_{n=1}^N (\ee_h^u, A \dot{\ee}_h^v)_{I_n} + \sum_{n=0}^{N-1} \left( \ee_h^u(t_{n+1}^-) \cdot A\ee_h^v(t_{n+1}^-) - \ee_h^u(t_{n}^+) \cdot A\ee_h^v(t_{n}^+) \right).
\end{equation}
Moreover, we note that
\begin{equation}\label{eq::simplify_u}
    \sum_{n=1}^N -(\dot{\ee}_\pi^u,A \ee_h^v)_{I_n}  - \sum_{n=0}^{N-1} [\ee_\pi^u]_{n}\cdot A \ee_h^v(t_{n}^+) = 0,
\end{equation}
thanks to \eqref{Eq:Projector}. 
%
The same holds for 
\begin{equation}\label{eq::simplyfy_v}
    \sum_{n=1}^N -(\dot{\ee}_\pi^v,A \ee_h^u)_{I_n}  - \sum_{n=0}^{N-1} [\ee_\pi^v]_{n}\cdot A \ee_h^u(t_{n}^+) = 0.
\end{equation}
Using \eqref{eq::simplify_lhs}, \eqref{eq::simplify_u}, and \eqref{eq::simplyfy_v} into \eqref{eq::starting} and rearranging the terms we get 
\begin{multline}\label{eq::startig_1}
\sum_{n=1}^{N} (M^{-1}A\ee_h^u,A \ee_h^u)_{I_n}
+ \sum_{n=1}^{N}(M^{-1}D\ee_h^v,A \ee_h^u)_{I_n} 
- \sum_{n=1}^{N} (\ee_h^v,A \ee_h^v)_{I_n} + \sum_{n=0}^{N-1} A[\ee_h^u]_{n}\cdot \ee_h^v(t_{n}^+)  \\  + \sum_{n=0}^{N-1} \left(A \ee_h^u(t_{n+1}^-) \cdot \ee_h^v(t_{n+1}^-) - A\ee_h^u(t_{n}^+) \cdot \ee_h^v(t_{n}^+) \right) + \sum_{n=0}^{N-1} A  [\ee_h^v]_{n}\cdot \ee_h^u(t_{n}^+) \\ = \sum_{n=1}^{N} (\ee_\pi^v,A \ee_h^v)_{I_n}  - (M^{-1}D\ee_\pi^v,A \ee_h^u)_{I_n}  + (M^{-1}A\ee_\pi^u,A \ee_h^u)_{I_n},
\end{multline}
where we also used that $A$ is symmetric  to cancel out the terms $-\sum_{n=1}^N (\ee_h^u, A \dot{\ee}_h^v)_{I_n} + \sum_{n=1}^N (\dot{\ee}_h^v , A \ee_h^u )_{I_n}$. 
Next, we consider the last three term on the left-hand side and manipulate them to obtain 
\begin{multline*}
\sum_{n=0}^{N-1} A[\ee_h^u]_{n}\cdot \ee_h^v(t_{n}^+)  + \sum_{n=0}^{N-1} \left(A \ee_h^u(t_{n+1}^-) \cdot \ee_h^v(t_{n+1}^-) - A\ee_h^u(t_{n}^+) \cdot \ee_h^v(t_{n}^+) \right) + \sum_{n=0}^{N-1} A  [\ee_h^v]_{n}\cdot \ee_h^u(t_{n}^+) \\
 = A \ee_h^u(0^+) \cdot \ee_h^v(0^+) + \sum_{n=0}^{N-1}  A  [\ee_h^v]_n\cdot [\ee_h^u]_n + A \ee_h^u(T^-) \cdot \ee_h^v(T^-).
\end{multline*}
Thus, equation \eqref{eq::startig_1} becomes 
\begin{multline*}
 \sum_{n=1}^{N} (M^{-1}A\ee_h^u,A \ee_h^u)_{I_n}
 =  - \sum_{n=1}^{N}(M^{-1}D\ee_h^v,A \ee_h^u)_{I_n} 
+ \sum_{n=1}^{N} (\ee_h^v,A \ee_h^v)_{I_n} 
\\ - A \ee_h^u(0^+) \cdot \ee_h^v(0^+) - \sum_{n=0}^{N-1}  A  [\ee_h^v]_n\cdot [\ee_h^u]_n - A \ee_h^u(T^-) \cdot \ee_h^v(T^-) \\ + \sum_{n=1}^{N} (\ee_\pi^v,A \ee_h^v)_{I_n}  - (M^{-1}D\ee_\pi^v,A \ee_h^u)_{I_n}  + (M^{-1}A\ee_\pi^u,A \ee_h^u)_{I_n}.
\end{multline*}
Finally, we use Cauchy-Schwarz and Young inequalities to have
\begin{multline*}
 \sum_{n=1}^{N} \| A\ee_h^u \|^2_{0,I_n}
 \lesssim  \sum_{n=1}^{N} \frac{1}{2\epsilon_1} \| \ee_h^v \|^2_{0,I_n} + \frac{\epsilon_1}{2} \| A \ee_h^u\|^2_{0,I_n} 
+ \sum_{n=1}^{N} \| \ee_h^v \|^2_{0,I_n} 
\\ +\frac12 (A^{\frac12}\ee_h^u(0^+))^2 + \frac12 \sum_{n=0}^{N-1}  (A^\frac12 [\ee_h^v]_n)^{\frac12}+\frac12 (A^{\frac12} \ee_h^u(T^-))^{\frac12} 
\\
+\frac12 (A^{\frac12}\ee_h^u(0^+))^2 + \frac12 \sum_{n=0}^{N-1}  (A^\frac12 [\ee_h^v]_n)^{\frac12}+\frac12 (A^{\frac12} \ee_h^u(T^-))^{\frac12} 
\\ + \sum_{n=1}^{N} \frac12 \| \ee_\pi^v \|^2_{0,I_n} + \frac12 \| \ee_h^v \|^2_{0,I_n}  + \sum_{n=1}^{N}  \frac{1}{2\epsilon_2}\|\ee_\pi^v \|_{0,I_n}^2 + \frac{\epsilon_2}{2} \| A \ee_h^u \|^2_{0,I_n} \\
+ \sum_{n=1}^{N} \frac{1}{2\epsilon_3} \| \ee_\pi^u\|_{0,I_n}^2 + \frac{\epsilon_3}{2} \| A \ee_h^u\|^2_{0,I_n},
\end{multline*}
for any positive $\epsilon_1,\epsilon_2$, and  $\epsilon_3$.
Now, choosing $\epsilon_1,\epsilon_2$, and  $\epsilon_3$ small enough we obtain 
\begin{equation*}
 \sum_{n=1}^{N} \| A\ee_h^u \|^2_{0,I_n}
 \lesssim  | (\ee_h^u, \ee_h^v) |^2_{\calB} + 
 \sum_{n=1}^{N} \|\ee_\pi^v \|_{0,I_n}^2 +  \| \ee_\pi^u\|_{0,I_n}^2,
\end{equation*}
which concludes the proof.
\end{proof}

\begin{myth}[\textit{Convergence of dG1}]\label{thm::conv_dgII}
Let Assumption~\ref{ass:matrices} holds.  Let $(\uu,\vvv)$ be the solution of \eqref{Eq:FirstOrderSystem1} and let $(\uu_{dG},\vvv_{dG}) \in \mathcal{V}_{dG}\times \mathcal{V}_{dG}$ be its dG finite element approximation based on employing formulation \eqref{eq::wf_dgII}. If $(\uu, \vvv)_{|_{I_n}} \in \bm H^{s_n} (I_n) \times \bm H^{s_n} (I_n)$, for any $n = 1, \ldots, N$, with $s_n \geq 2$,
then
\begin{equation}\label{eq::error_dgII}
    | (\uu,\vvv) - (\uu_{dG},\vvv_{dG}) |_{\calB} \lesssim \sum_{n=1}^N {\Delta t}_n^{\mu_n+\frac12} \left( (r_n+2) \frac{(r_n-\mu_n)!}{(r_n+\mu_n)!} \right)^{\frac12} \| (\uu,\vvv) \|_{\bm H^{\mu_n+1}(I_n)\times\bm H^{\mu_n+1}(I_n)},   
\end{equation}
where $s_n = min(\mu_n,r_n)$ for any $n = 1, \ldots, N$, and where the hidden constant depends on the norm of
the matrices $M,D$ and $A$.
\end{myth}

\begin{proof}
The proof closely follows the arguments in \cite{AntoniettiMiglioriniMazzieri2021} and we only report here the sketch for completeness. 
We start observing that $|(\ee^u,\ee^v)|_{\calB} \le |(\ee_\pi^u,\ee_\pi^v)|_{\calB}  + |(\ee_h^u,\ee_h^v|_{\calB}$. 
Employing the properties of the projector \eqref{Eq:Projector} and estimates \eqref{Eq:ProjectionError} and \eqref{Eq:DerivativeProjectionError}, we can bound $|(\ee_\pi^u,\ee_\pi^v)|_{\calB}$ as
	\begin{align*}
|(\ee_\pi^u,\ee_\pi^v)|_{\calB}^2 & = \calB((\ee_\pi^u,\ee_\pi^v),(A\ee_\pi^u,\ee_\pi^v)) \\ & \lesssim \sum_{n=1}^N \bigg(\frac{\Delta t_n}{2}\bigg)^{2\mu_n+1} (r_n+2) \frac{(r_n-\mu_n)!}{(r_n+\mu_n)!}\| (\uu,\vvv) \|^2_{\bm H^{s_n}(I_n)\times\bm H^{s_n}(I_n)},
	\end{align*}
	where $\mu_n = \min(r_n,s_n)$, for any $n=1,\dots,N$.
	For the term $|(\ee_h^u,\ee_h^v)|_{\calB}$ we use  \eqref{eq::consistency_B} and integrate by parts to get
	\begin{equation*}
	\begin{split}
		|(\ee_h^u,\ee_h^v)|^2_{\calB} &= \calB((\ee_h^u,\ee_h^v),(A\ee_h^u,\ee_h^v))) = -\calB((\ee_\pi^u,\ee_\pi^v),(A\ee_h^u,\ee_h^v)) \\ 
		& = -\sum_{n=1}^N(\ee_\pi^v,A\ee_h^u)_{I_n}  -\sum_{n=1}^N(A\ee_\pi^v,\ee_h^v)_{I_n} +\sum_{n=1}^N(D\ee_\pi^v,\ee_h^v)_{I_n}.
	\end{split}
	\end{equation*}
	Thus, we employ the Cauchy-Schwarz and inequalities, together with \eqref{eq::estimateAu} to obtain for any $\epsilon>0$
	\begin{equation*}
		|(\ee_h^u,\ee_h^v)|^2_{\calB}  \lesssim  \frac{1}{2\epsilon}\sum_{n=1}^N \Big(||\ee_\pi^u||_{L^2(I_n)}^2 + ||\ee_\pi^v||_{L^2(I_n)}^2\Big) + \frac{\epsilon}{2} |(\ee_h^u,\ee_h^v)|^2_{\calB},
 	\end{equation*}
from which the thesis follows.
\end{proof}
\begin{remark}
   For the undamped case, i.e., $D=0$ for case a) in Assumption~\ref{ass:matrices}, the proof 
   holds only for constant and linear polynomials $r_n=0,1$ for $n=1,\ldots,N$. See \cite{rezaei2023discontinuous}.
\end{remark}

\section{Algebraic formulation}\label{sec::AlgebraicFormulation}
In this section, we briefly review the algebraic formulation stemming from discretization  
 \textit{dG2} in \eqref{eq::wf_dgI}  or \textit{dG1} in \eqref{eq::wf_dgII} for the single time interval $I_n$. 
 In practice, we compute the numerical solution on the time interval at a time slab, assuming the initial conditions from the previous time interval.
In the following,  we introduce a basis $\{\psi^{\ell}(t)\}_{{\ell}=1,\dots,r_n+1}$ for the polynomial space $\mathcal{P}^{r_n}(I_n)$ and define a vectorial basis $\{ \boldsymbol{\Psi}_i^{\ell}(t) \}_{i=1,\dots,{N_h}}^{{\ell}=1,\dots,r_n+1}$ of $V_n^{r_n}$.
Then, we set $D_n={N_h}(r_n+1)$ and write the trial functions $\uu_{dG}, \vvv_{dG} \in V_n^{r_n}$ as
\begin{equation*}
	\uu_{dG}(t) = \sum_{j=1}^{{N_h}} \sum_{m=1}^{r_n+1} \alpha_{j}^m \boldsymbol{\Psi}_j^m(t), \quad 	\vvv_{dG}(t) = \sum_{j=1}^{{N_h}} \sum_{m=1}^{r_n+1} \beta_{j}^m \boldsymbol{\Psi}_j^m(t),
\end{equation*}
where $\alpha_{j}^m,\beta_{j}^m\in\mathbb{R}$ for $j=1,\dots,{N_h}$, $m=1,\dots,r_n+1$.

\subsection{Formulation \textit{dG2}}
We consider Problem~\eqref{eq::wf_dgI} on $I_n$ to  get: 
 find $\uu_{dG}\in V^{r_n}_n$ such that
\begin{multline}
	\label{eq::weak_reformulated_I}
	(M\ddot{\uu}_{dG},\dot{\ww})_{I_n} + (D\dot{\uu}_{dG},\dot{\ww})_{I_n} + (A\uu_{dG},\dot{\ww})_{I_n}  \\ +  M\dot{\uu}_{dG}(t_{n-1}^+)\cdot\dot{\ww}(t_{n-1}^+) + A\uu_{dG}(t_{n-1}^+)\cdot\ww(t_{n-1}^+) \\ =  (\FF,\dot{\ww})_{I_n} + M\dot{\uu}_{dG}(t_{n-1}^-)\cdot\dot{\ww}(t_{n-1}^+) + A\uu_{dG}(t_{n-1}^-)\cdot\ww(t_{n-1}^+),
\end{multline}
 for any $\ww \in V^{r_n}_n$.
Writing \eqref{eq::weak_reformulated_I}  for any test function $\boldsymbol{\Psi}_i^{\ell}(t)$, $i=1,\dots,d$, $\ell=1\,\dots,r_n+1$ we obtain the linear system of the form
\begin{equation}
	\label{Eq:LinearSystem_I}
	M_n\UU_n = \bm F_n,
\end{equation}
where $\UU_n,\bm F_n \in \mathbb{R}^{D_n}$ are the vectors of the expansion coefficient corresponding to the numerical solution and the right-hand side on the interval $I_n$ in the chosen basis, respectively.
We next investigate the structure of the matrix $M_n$ by deﬁning the following local matrices for $\ell,m = 1, . . . , r_n + 1$, namely
\begin{equation*}
    (N_1)_{\ell m} = (\ddot{\psi}^m,\dot{\psi}^\ell)_{I_n}, \quad  (N_2)_{\ell m} = (\dot{\psi}^m,\dot{\psi}^\ell)_{I_n}, \quad (N_3)_{\ell m} = ({\psi}^m,\dot{\psi}^\ell)_{I_n},
\end{equation*}
and 
\begin{equation*}
    (N_4)_{\ell m} = \dot{\psi}^m (t_{n-1}^+) \dot{\psi}^{\ell}(t_{n-1}^+), \quad  (N_5)_{\ell m} = \psi^m (t_{n-1}^+) \psi^{\ell}(t_{n-1}^+).
\end{equation*}
Then, we have 
$$ M_n = M \otimes ( N_1 + N_4 ) + D \otimes N_2 + A \otimes ( N_3 + N_5 ),$$
where $A \otimes B$ denotes the Kronecher tensor product between the matrix A and the matrix B.

\subsection{Formulation \textit{dG1}}

 Proceeding similarly for \eqref{eq::wf_dgII} we obtain: 
 find $(\uu_{dG}, \vvv_{dG}) \in V^{r_n}_n \times V^{r_n}_n$ such that
\begin{equation}
	\label{eq::weak_reformulated_II}
\begin{cases}
(\dot{\uu}_{dG},\ww)_{I_n} - (\vvv_{dG},\ww)_{I_n} + \uu_{dG}(t_{n-1}^+)\cdot\ww(t_{n-1}^+) = \uu_{dG}(t_{n-1}^-)\cdot\ww(t_{n-1}^+) & \\
(M\dot{\vvv}_{dG},\zz)_{I_n} + (D\vvv_{dG},\zz)_{I_n} + (A\uu_{dG},\zz)_{I_n}  +  M\vvv_{dG}(t_{n-1}^+)\cdot\zz(t_{n-1}^+)  & \\ \qquad\qquad\qquad\qquad\qquad\qquad\qquad\qquad\qquad  = (\FF,\zz)_{I_n} + M\vvv_{dG}(t_{n-1}^-)\cdot\zz(t_{n-1}^+), & \\    
\end{cases}
\end{equation}
for any $(\ww, \zz) \in V^{r_n}_n \times V^{r_n}_n$.

Writing \eqref{eq::weak_reformulated_II}  for any test function $\boldsymbol{\Psi}_i^{\ell}(t)$, $i=1,\dots,d$, $\ell=1\,\dots,r_n+1$ we obtain the following linear system 
\begin{equation*}
	\label{Eq:StiffnessMatrix}
 \begin{bmatrix}
		I_d \otimes (L_1 + L_3) & - I_d \otimes L_2 \\
		A \otimes L_2 & M \otimes (L_1+L_3) + D \otimes L_2 
	\end{bmatrix}\begin{bmatrix}
	    \UU_n,\\ \vvv_n
	\end{bmatrix} = \begin{bmatrix}
	    \GG_n^u,\\ \GG_n^v
	\end{bmatrix},
\end{equation*}
where $\UU_n,\VV_n \in \mathbb{R}^{D_n}$ and $\GG_n^u,\GG_n^v \in \mathbb{R}^{D_n} $ are the {vectors of the expansion coefficients} of the solution and of the right-hand side in \eqref{eq::weak_reformulated_II}, respectively.
 The local time matrices $L_1,L_2$, and $L_3 \in \mathbb{R}^{r_n+1 \times r_n+1}$ are defined as
\begin{equation*}
	\label{Eq:TimeMatrices}
	(L_1)_{{\ell}m} = (\dot{\psi}^m,\psi^{\ell})_{I_n}, \qquad (L_2)_{{\ell}m} = (\psi^m,\psi^{\ell})_{I_n}, \qquad (L_3)_{{\ell}m} = \psi^m(t_{n-1}^+) \psi^{\ell}(t_{n-1}^+),
\end{equation*}
for $\ell,m=1,\ldots,r_n+1$. Similarly to \cite{ThHe2005}, and following \cite{AntoniettiMiglioriniMazzieri2021} we define
\begin{equation*}
L_4 = (L_1+L_3)^{-1}, \qquad L_5 = L_4L_2, \qquad L_6 = L_2L_4, \qquad L_7 = L_2L_4L_2.
\end{equation*}
Next, we apply a block Gaussian elimination to get
\begin{equation}\label{eq:linear_system_to_solve}
	\begin{bmatrix}
	I_{D_n} & -I_{d} \otimes L_5 \\
	0 & \widehat{M}_n
	\end{bmatrix} 
	\begin{bmatrix}
	\UU_n \\
	\VV_n
	\end{bmatrix} = 
	\begin{bmatrix}
	(I_d \otimes L_4)\GG_n^u \\
	\GG_n^v - ({I}_{d}\otimes L_6) \GG_n^u
	\end{bmatrix},
\end{equation}
where
\begin{equation}\label{eq::linear_system_dGII}
	\widehat{M}_n = M \otimes (L_1+L_3) + D \otimes L_2 + A \otimes L_7.
\end{equation}
As a solution strategy, we first compute  $\VV_n$, by solving the second equation of \eqref{eq:linear_system_to_solve} and then update  
$	\UU_n = ({I}_{d} \otimes L_5) \VV_n + ({I}_{d}\otimes L_4)\GG_n^u.
$
\\
We notice that, despite doubling the dimension of the problem at hand, formulation \eqref{eq::weak_reformulated_II} requires at each time interval the solution of a linear system of the same size of \eqref{eq::weak_reformulated_I}, cf \eqref{eq:linear_system_to_solve}. 
Finally, we recall that the solution strategy described above remains valid also for the formulations presented in \cite{rezaei2023discontinuous,AntoniettiMiglioriniMazzieri2021} where the first equation in \eqref{eq::weak_reformulated_II} is multiplied by the positive definite matrix $A$. 
\section{Numerical results}\label{sec::numerical_results}

In the following experiments, we employ the proposed dG methods to solve \eqref{Eq:SecondOrderEquation} which is obtained from the spatial discretization of different wave propagation problems with the Polytopal Discontinuous Galerkin (PolyDG) method implemented in the library \texttt{lymph} (\url{https://lymph.bitbucket.io/}), see \cite{antonietti2024lymph}.
In the following, we will address the cases of wave propagation in acoustic, poroelastic, and coupled poroelastic-acoustic media. 
We suppose $\Omega \subset \mathbb{R}^2$ to be an open bounded domain, having sufficiently regular boundary $\Gamma$. 

\subsection{Wave propagation in a fluid medium}\label{sec::acoustic_example} 
\begin{figure}[!htbp]
    \centering
    \includegraphics[width=0.45\textwidth]{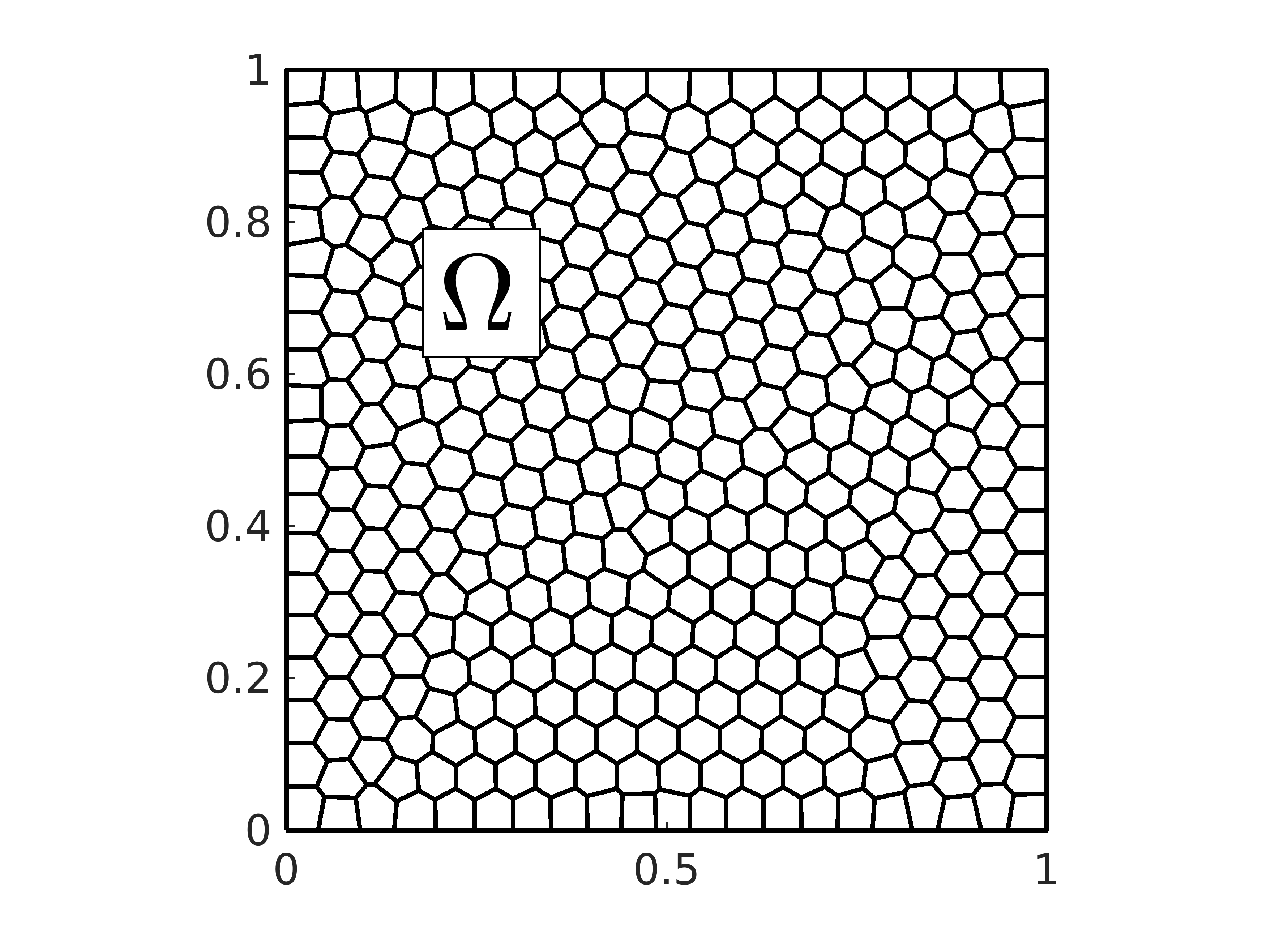}
    \includegraphics[width=0.45\textwidth]{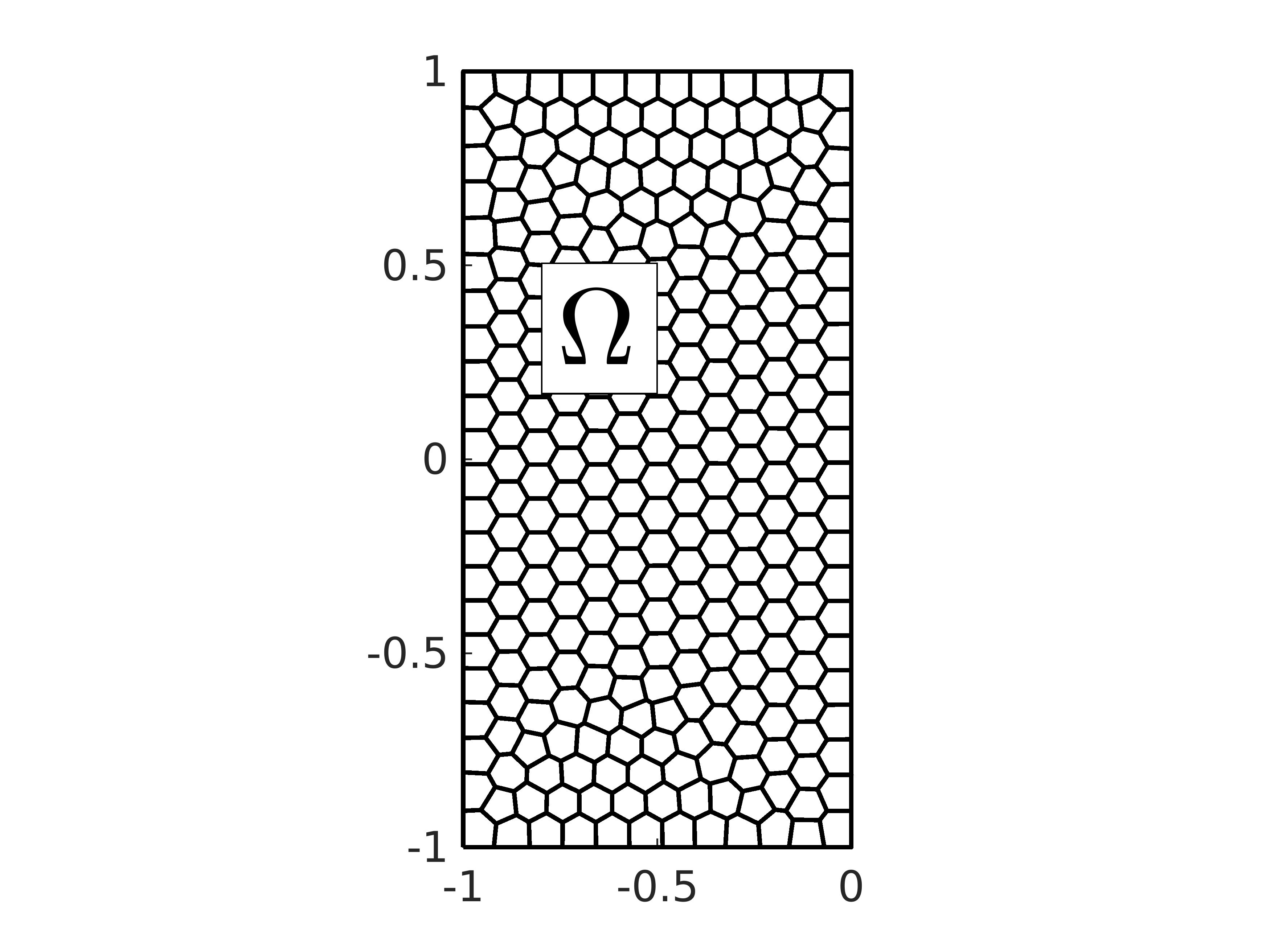}
    \caption{Polygonal meshes for the test case of Section~\ref{sec::acoustic_example} (left) and Section~\ref{sec::poro_example} (right).}
    \label{fig:acoustic-poro-domain}
\end{figure}
In a fluid domain $\Omega$ we consider the following model
\begin{equation}\label{eq::acoustic}
\begin{cases}
c^{-2}\ddot{\varphi}_a - \rho_a^{-1}  \nabla\cdot(\rho_a \nabla \varphi_a)=f_a,  & \text{in }\Omega\times(0,T],\\
  \varphi_a = g_a, & \text{on }\Gamma\times(0,T],\\  (\varphi_a, \dot{\varphi}_a)(0) = (\varphi_0, \psi_0), & \text{in }\Omega,
\end{cases}
\end{equation}
being $\varphi_a$ the acoustic potential, $\rho_a >0$ the medium density, and  $f_a, g_a, \varphi_0$ and $\psi_0$ regular enough data.
It is well known that, provided the penalty parameter is chosen large enough, the space discretization by PolydG leads to system \eqref{Eq:SecondOrderEquation} in which $M,A$ are positive definite and $D=0$, i.e., case a) of Assumption~\ref{ass:matrices}. We refer to \cite{AntoniettiBonaldi2020} for more details.  We take $\Omega=(0,1)^2$ and patition it into $400$ polygonal elements, cf. Figure~\ref{fig:acoustic-poro-domain} (left), and fix the space polynomial degree equal to $7$. Next, we consider the following manufactured solution
\begin{align*}
 \varphi_{ex}(x,t) & = \sin(\sqrt{2} \pi t) x^2\sin(\pi x)\sin(\pi y),\\
 \psi_{ex}(x,t)  & = \dot{\varphi}_{ex}(x,t), 
\end{align*}
 compute the external force $f_a$, the boundary data $g_a$, and initial conditions $\phi_0,\psi_0$ accordingly, and set $\rho_a=c_a=1$.    
For the time integration, we compare methods \textit{dG2} and \textit{dG1} described in Sections ~\ref{sec::dg_form_I} and \ref{sec::dg_form_II}, respectively. 
In Figure~\ref{fig:Energy_conv_test_dt} we report the computed energy errors $|e_\varphi|_{\calA} = |\varphi_{ex} - \varphi_{dG}|_{\calA}$, for \textit{dG2},
and $|(e_\varphi,e_\psi)|_{\calB} = |(\varphi_{ex},\psi_{ex}) - (\varphi_{dG},\psi_{dG})|_{\calB}$, for \textit{dG1}, as a function of the time step $\Delta t$. We choose a time polynomial degree $r=1,..,4$, and a final time $T=0.6$. 
It is possible to see that the results agree with the theoretical results of Section~\ref{sec:Convergence}. 
We remark that the order of convergence is verified also for higher-order polynomials ($r=3,4$), despite these cases are not explicitly addressed in our theoretical results of   Theorem~\ref{thm::conv_dgI} and Theorem~\ref{thm::conv_dgII}. 
 In Figure~\ref{fig:L2conv_test_dt} we also compare the accuracy of the solutions obtained with both \textit{dG2} and \textit{dG1} methods concerning the $L^2$-norm at final time $T$. We can observe 
an order of convergence of $\mathcal{O}(\Delta t^{r+\frac12})$ (resp. $\mathcal{O}(\Delta t^{r+\frac32})$) for the \textit{dG2} (resp. \textit{dG1}) method.
This is in agreement with results reported in \cite{CiaramellaGanderMazzieri2025}.
\begin{figure}[!htbp]
\centering
%
%
\begin{tikzpicture}

\begin{axis}[%
width=0.33\textwidth,
height=0.33\textwidth,
scale only axis,
xmode=log,
xmin=0.01,
xmax=0.5,
xminorticks=true,
xlabel style={font=\color{black}},
xlabel={$\Delta t $},
ymode=log,
ymin=1e-09,
ymax=1,
yminorticks=true,
ylabel style={font=\color{black}},
ylabel={$| e_\varphi |_{\mathcal{A}}$},
axis background/.style={fill=white},
title ={$dG2$},
xmajorgrids,
xminorgrids,
ymajorgrids,
yminorgrids,
legend style={at={(0.65,0.05)}, anchor=south west, legend cell align=left, align=left, draw=black}
]
\addplot [color=blue, line width=2.0pt, mark=o, mark options={solid, blue}]
  table[row sep=crcr]{%
     2.000000000000000e-01     2.628074134353042e-01 \\
     1.000000000000000e-01     3.296270132879754e-01 \\
     5.000000000000000e-02     2.782793473302495e-01 \\
     2.000000000000000e-02     1.910666900967379e-01 \\
};
\addlegendentry{$r=1$}

\addplot [color=red, line width=2.0pt, mark=asterisk, mark options={solid, red}]
  table[row sep=crcr]{%
     2.000000000000000e-01     6.063459234531784e-02\\
     1.000000000000000e-01     2.299624571395244e-02\\
     5.000000000000000e-02     8.607811509773690e-03\\
     2.000000000000000e-02     2.416657827371976e-03 \\
};
\addlegendentry{$r=2$}

\addplot [color=green, line width=2.0pt, mark=square, mark options={solid, green}]
  table[row sep=crcr]{%
     2.000000000000000e-01     6.461516801054583e-03\\
     1.000000000000000e-01     1.349965008894389e-03\\
     5.000000000000000e-02     2.490685147155186e-04\\
     2.000000000000000e-02     2.581392948773486e-05\\
};
\addlegendentry{$r=3$}

\addplot [color=cyan, line width=2.0pt, mark=diamond, mark options={solid, cyan}]
  table[row sep=crcr]{%
     2.000000000000000e-01     3.622737861344061e-04\\
     1.000000000000000e-01     3.415051607552126e-05\\
     5.000000000000000e-02     3.217612062356779e-06\\
     2.000000000000000e-02     1.479781315881082e-07\\
};
\addlegendentry{$r=4$}

\addplot [color=black, forget plot]
  table[row sep=crcr]{%
0.05	0.278\\
0.02	0.176\\
};

\addplot [color=black, forget plot]
  table[row sep=crcr]{%
0.05	0.176\\
0.02	0.176\\
};
\addplot [color=black, forget plot]
  table[row sep=crcr]{%
0.05	0.278\\
0.05	0.176\\
};

\addplot [color=black, forget plot]
  table[row sep=crcr]{%
0.05	8.607811509773690e-03 \\
0.02	2.177623203223807e-03\\
};
\addplot [color=black, forget plot]
  table[row sep=crcr]{%
0.05	2.177623203223807e-03\\
0.02	2.177623203223807e-03\\
};
\addplot [color=black, forget plot]
  table[row sep=crcr]{%
0.05	8.607811509773690e-03 \\
0.05	2.177623203223807e-03\\
};

\addplot [color=black, forget plot]
  table[row sep=crcr]{%
0.05	2.490685147155186e-04\\
0.02	2.520396159795851e-05\\
};
\addplot [color=black, forget plot]
  table[row sep=crcr]{%
0.05	2.520396159795851e-05\\
0.02	2.520396159795851e-05\\
};
\addplot [color=black, forget plot]
  table[row sep=crcr]{%
0.05	2.490685147155186e-04\\
0.05	2.520396159795851e-05\\
};

\addplot [color=black, forget plot]
  table[row sep=crcr]{%
0.05	3.217612062356779e-06 \\
0.02	1.302397791216530e-07\\
};
\addplot [color=black, forget plot]
  table[row sep=crcr]{%
0.05	1.302397791216530e-07\\
0.02	1.302397791216530e-07\\
};
\addplot [color=black, forget plot]
  table[row sep=crcr]{%
0.05	3.217612062356779e-06 \\
0.05	1.302397791216530e-07\\
};

\node[right, align=left, text=black, font=\normalsize]
at (axis cs:0.05,0.2) {$0.5$};

\node[right, align=left, text=black, font=\normalsize]
at (axis cs:0.05,0.005) {$1.5$};

\node[right, align=left, text=black, font=\normalsize]
at (axis cs:0.05,0.0001) {$2.5$};

\node[right, align=left, text=black, font=\normalsize]
at (axis cs:0.05,0.0000008) {$3.5$};

\end{axis}

\end{tikzpicture}%
%
%
\begin{tikzpicture}

\begin{axis}[%
width=0.33\textwidth,
height=0.33\textwidth,
scale only axis,
xmode=log,
xmin=0.01,
xmax=0.5,
xminorticks=true,
xlabel style={font=\color{black}},
xlabel={$\Delta t $},
ymode=log,
ymin=1e-09,
ymax=1,
yminorticks=true,
ylabel style={font=\color{black}},
ylabel={$|(e_\varphi,e_\psi)|_{\calB}$},
axis background/.style={fill=white},
title ={$dG1$},
xmajorgrids,
xminorgrids,
ymajorgrids,
yminorgrids,
legend style={at={(0.65,0.05)}, anchor=south west, legend cell align=left, align=left, draw=black}
]
\addplot [color=blue, line width=2.0pt, mark=o, mark options={solid, blue}]
  table[row sep=crcr]{%
     2.000000000000000e-01     2.342658529176124e-01 \\
     1.000000000000000e-01     1.249082459270910e-01 \\
     5.000000000000000e-02     4.917641987742438e-02 \\
     2.000000000000000e-02     1.308057504395111e-02 \\
};
\addlegendentry{$r=1$}

\addplot [color=red, line width=2.0pt, mark=asterisk, mark options={solid, red}]
  table[row sep=crcr]{%
     2.000000000000000e-01     1.823047354888938e-02\\
     1.000000000000000e-01     4.267262215448603e-03\\
     5.000000000000000e-02     8.370799345739045e-04\\
     2.000000000000000e-02     9.173791081965823e-05\\
};
\addlegendentry{$r=2$}

\addplot [color=green, line width=2.0pt, mark=square, mark options={solid, green}]
  table[row sep=crcr]{%
     2.000000000000000e-01     9.428013038118890e-04\\
     1.000000000000000e-01     1.294019804122741e-04\\
     5.000000000000000e-02     1.295639162176570e-05\\
     2.000000000000000e-02     5.659586736462439e-07\\
};
\addlegendentry{$r=3$}

\addplot [color=cyan, line width=2.0pt, mark=diamond, mark options={solid, cyan}]
  table[row sep=crcr]{%
     2.000000000000000e-01     6.735485189797528e-05\\
     1.000000000000000e-01     3.328985044423858e-06\\
     5.000000000000000e-02     1.550458527983530e-07\\
     2.000000000000000e-02     2.767015673443082e-09\\
};
\addlegendentry{$r=4$}

\addplot [color=black, forget plot]
  table[row sep=crcr]{%
0.05	4.917641987742438e-02\\
0.02	1.244075951883515e-02\\
};

\addplot [color=black, forget plot]
  table[row sep=crcr]{%
0.05	1.244075951883515e-02\\
0.02	1.244075951883515e-02\\
};
\addplot [color=black, forget plot]
  table[row sep=crcr]{%
0.05	4.917641987742438e-02\\
0.05	1.244075951883515e-02\\
};

\addplot [color=black, forget plot]
  table[row sep=crcr]{%
0.05	8.370799345739045e-04 \\
0.02	8.470653366010451e-05\\
};
\addplot [color=black, forget plot]
  table[row sep=crcr]{%
0.05	8.470653366010451e-05\\
0.02	8.470653366010451e-05\\
};
\addplot [color=black, forget plot]
  table[row sep=crcr]{%
0.05	8.470653366010451e-05 \\
0.05	8.370799345739045e-04\\
};

\addplot [color=black, forget plot]
  table[row sep=crcr]{%
0.05	1.295639162176570e-05\\
0.02	5.244378596083509e-07\\
};
\addplot [color=black, forget plot]
  table[row sep=crcr]{%
0.05	5.244378596083509e-07\\
0.02	5.244378596083509e-07\\
};
\addplot [color=black, forget plot]
  table[row sep=crcr]{%
0.05	5.244378596083509e-07\\
0.05	1.295639162176570e-05\\
};

\addplot [color=black, forget plot]
  table[row sep=crcr]{%
0.05	1.550458527983530e-07\\
0.02	2.510325947422651e-09\\
};
\addplot [color=black, forget plot]
  table[row sep=crcr]{%
0.05	2.510325947422651e-09\\
0.02	2.510325947422651e-09\\
};
\addplot [color=black, forget plot]
  table[row sep=crcr]{%
0.05	2.510325947422651e-09\\
0.05	1.550458527983530e-07\\
};

\node[right, align=left, text=black, font=\normalsize]
at (axis cs:0.05,0.02) {$1.5$};

\node[right, align=left, text=black, font=\normalsize]
at (axis cs:0.05,0.0003) {$2.5$};

\node[right, align=left, text=black, font=\normalsize]
at (axis cs:0.05,0.000003) {$3.5$};

\node[right, align=left, text=black, font=\normalsize]
at (axis cs:0.05,0.00000003) {$4.5$};

\end{axis}

\end{tikzpicture}%
\caption{Test case of Section~\ref{sec::acoustic_example}. 
Computed errors $|e_\varphi|_{\calA}$ for the \textit{dG2} (left), and  $|(e_\varphi,e_\psi)|_{\calB}$ for the \textit{dG1} (right) in logarithmic scale as a function of the time step $\Delta t$ for different polynomial degrees $r=1,2,3,4$. We set $N_{el}=400$ polygonal elements and a space polynomial degree equal to $7$.}
\label{fig:Energy_conv_test_dt}
\end{figure}
\begin{figure}[!htbp]
\centering
%
%
\begin{tikzpicture}

\begin{axis}[%
width=0.31\textwidth,
height=0.31\textwidth,
scale only axis,
xmode=log,
xmin=0.01,
xmax=0.5,
xminorticks=true,
xlabel style={font=\color{black}},
xlabel={$\Delta t $},
ymode=log,
ymin=1e-16,
ymax=1,
yminorticks=true,
ylabel style={font=\color{black}},
ylabel={$\| e_\varphi(T) \|_0 $},
axis background/.style={fill=white},
title ={$dG2$},
xmajorgrids,
xminorgrids,
ymajorgrids,
yminorgrids,
legend style={at={(0.65,0.05)}, anchor=south west, legend cell align=left, align=left, draw=black}
]
\addplot [color=blue, line width=2.0pt, mark=o, mark options={solid, blue}]
  table[row sep=crcr]{%
     2.000000000000000e-01     3.610161928807072e-02\\
     1.000000000000000e-01     2.001775336144223e-02\\
     5.000000000000000e-02     1.001685286449612e-02\\
     2.000000000000000e-02     3.603545092321021e-03\\
};
\addlegendentry{$r=1$}

\addplot [color=red, line width=2.0pt, mark=asterisk, mark options={solid, red}]
  table[row sep=crcr]{%
     2.000000000000000e-01     1.378156398327358e-03 \\
     1.000000000000000e-01     1.494969159903137e-04 \\
     5.000000000000000e-02     1.695570687628642e-05 \\
     2.000000000000000e-02     1.165037673849501e-06 \\
};
\addlegendentry{$r=2$}

\addplot [color=green, line width=2.0pt, mark=square, mark options={solid, green}]
  table[row sep=crcr]{%
     2.000000000000000e-01     8.465537531723704e-06\\
     1.000000000000000e-01     2.755916842454628e-07\\
     5.000000000000000e-02     1.207318415062316e-08\\
     2.000000000000000e-02     2.856729581048892e-10\\
};
\addlegendentry{$r=3$}

\addplot [color=cyan, line width=2.0pt, mark=diamond, mark options={solid, cyan}]
  table[row sep=crcr]{%
     2.000000000000000e-01     1.109343089955558e-07\\
     1.000000000000000e-01     8.126602350010218e-10\\
     5.000000000000000e-02     6.430321140445857e-12\\
     2.000000000000000e-02     1.784370010700440e-12\\
};
\addlegendentry{$r=4$}

\addplot [color=black, forget plot]
  table[row sep=crcr]{%
0.05	1.001685286449612e-02\\
0.02	4.006741145798448e-03\\
};

\addplot [color=black, forget plot]
  table[row sep=crcr]{%
0.05	4.006741145798448e-03\\
0.02	4.006741145798448e-03\\
};
\addplot [color=black, forget plot]
  table[row sep=crcr]{%
0.05	1.001685286449612e-02\\
0.05	4.006741145798448e-03\\
};

\addplot [color=black, forget plot]
  table[row sep=crcr]{%
0.05	1.695570687628642e-05 \\
0.02	1.715796898151805e-06\\
};
\addplot [color=black, forget plot]
  table[row sep=crcr]{%
0.05	1.715796898151805e-06\\
0.02	1.715796898151805e-06\\
};
\addplot [color=black, forget plot]
  table[row sep=crcr]{%
0.05	1.695570687628642e-05 \\
0.05	1.715796898151805e-06\\
};

\addplot [color=black, forget plot]
  table[row sep=crcr]{%
0.05	1.207318415062316e-08\\
0.02	4.886881347406662e-10\\
};
\addplot [color=black, forget plot]
  table[row sep=crcr]{%
0.05	4.886881347406662e-10\\
0.02	4.886881347406662e-10\\
};
\addplot [color=black, forget plot]
  table[row sep=crcr]{%
0.05	1.207318415062316e-08\\
0.05	4.886881347406662e-10\\
};

\addplot [color=black, forget plot]
  table[row sep=crcr]{%
0.05	 6.430321140445857e-12\\
0.02	1.041124397575188e-13\\
};
\addplot [color=black, forget plot]
  table[row sep=crcr]{%
0.05	1.041124397575188e-13\\
0.02	1.041124397575188e-13\\
};
\addplot [color=black, forget plot]
  table[row sep=crcr]{%
0.05	 6.430321140445857e-12\\
0.05	1.041124397575188e-13\\
};

\node[right, align=left, text=black, font=\normalsize]
at (axis cs:0.05,0.005) {$1$};

\node[right, align=left, text=black, font=\normalsize]
at (axis cs:0.05,0.000007) {$2.5$};

\node[right, align=left, text=black, font=\normalsize]
at (axis cs:0.05,0.000000003) {$3.5$};

\node[right, align=left, text=black, font=\normalsize]
at (axis cs:0.05,1.e-12) {$4.5$};

\end{axis}

\end{tikzpicture}%
%
%
\begin{tikzpicture}

\begin{axis}[%
width=0.31\textwidth,
height=0.31\textwidth,
scale only axis,
xmode=log,
xmin=0.01,
xmax=0.5,
xminorticks=true,
xlabel style={font=\color{black}},
xlabel={$\Delta t $},
ymode=log,
ymin=1e-16,
ymax=1,
yminorticks=true,
ylabel style={font=\color{black}},
ylabel={$\|  e_\varphi(T) \|_0 $},
axis background/.style={fill=white},
title ={$dG1$},
xmajorgrids,
xminorgrids,
ymajorgrids,
yminorgrids,
legend style={at={(0.65,0.05)}, anchor=south west, legend cell align=left, align=left, draw=black}
]
\addplot [color=blue, line width=2.0pt, mark=o, mark options={solid, blue}]
  table[row sep=crcr]{%
     2.000000000000000e-01     3.132871638013428e-02\\
     1.000000000000000e-01     1.189962884566634e-02\\
     5.000000000000000e-02     3.430657539125001e-03\\
     2.000000000000000e-02     6.008647594894234e-04\\
};
\addlegendentry{$r=1$}

\addplot [color=red, line width=2.0pt, mark=asterisk, mark options={solid, red}]
  table[row sep=crcr]{%
     2.000000000000000e-01     4.328106204533199e-04\\
     1.000000000000000e-01     4.099823072578736e-05\\
     5.000000000000000e-02     2.913434535096544e-06\\
     2.000000000000000e-02     8.073412490949890e-08\\
};
\addlegendentry{$r=2$}

\addplot [color=green, line width=2.0pt, mark=square, mark options={solid, green}]
  table[row sep=crcr]{%
     2.000000000000000e-01     4.745643831193803e-06\\
     1.000000000000000e-01     8.264836643450768e-08\\
     5.000000000000000e-02     1.341952400937620e-09\\
     2.000000000000000e-02     5.666921342022546e-12\\
};
\addlegendentry{$r=3$}

\addplot [color=cyan, line width=2.0pt, mark=diamond, mark options={solid, cyan}]
  table[row sep=crcr]{%
     2.000000000000000e-01     3.688271785860382e-08\\
     1.000000000000000e-01     2.119324600468031e-10\\
     5.000000000000000e-02     1.180672814229561e-12\\
     2.000000000000000e-02     1.991036109957030e-13\\
};
\addlegendentry{$r=4$}

\addplot [color=black, forget plot]
  table[row sep=crcr]{%
0.05	3.430657539125001e-03\\
0.02	3.471581342612229e-04\\
};

\addplot [color=black, forget plot]
  table[row sep=crcr]{%
0.05	3.471581342612229e-04\\
0.02	3.471581342612229e-04\\
};
\addplot [color=black, forget plot]
  table[row sep=crcr]{%
0.05	3.430657539125001e-03\\
0.05	3.471581342612229e-04\\
};

\addplot [color=black, forget plot]
  table[row sep=crcr]{%
0.05	2.913434535096544e-06  \\
0.02	1.179275384921452e-07\\
};
\addplot [color=black, forget plot]
  table[row sep=crcr]{%
0.05	1.179275384921452e-07\\
0.02	1.179275384921452e-07\\
};
\addplot [color=black, forget plot]
  table[row sep=crcr]{%
0.05	2.913434535096544e-06  \\
0.05	1.179275384921452e-07\\
};

\addplot [color=black, forget plot]
  table[row sep=crcr]{%
0.05	1.379113777810375e-09\\
0.02	2.232904033484885e-11\\
};
\addplot [color=black, forget plot]
  table[row sep=crcr]{%
0.05	2.232904033484885e-11\\
0.02	2.232904033484885e-11\\
};
\addplot [color=black, forget plot]
  table[row sep=crcr]{%
0.05	1.379113777810375e-09\\
0.05	2.232904033484885e-11\\
};

\addplot [color=black, forget plot]
  table[row sep=crcr]{%
0.05	1.180672814229561e-12\\
0.02	7.646444061504041e-15\\
};
\addplot [color=black, forget plot]
  table[row sep=crcr]{%
0.05	7.646444061504041e-15\\
0.02	7.646444061504041e-15\\
};
\addplot [color=black, forget plot]
  table[row sep=crcr]{%
0.05	1.180672814229561e-12\\
0.05	7.646444061504041e-15\\
};

\node[right, align=left, text=black, font=\normalsize]
at (axis cs:0.05,0.001) {$2.5$};

\node[right, align=left, text=black, font=\normalsize]
at (axis cs:0.05,0.0000007) {$3.5$};

\node[right, align=left, text=black, font=\normalsize]
at (axis cs:0.05,0.0000000003) {$4.5$};

\node[right, align=left, text=black, font=\normalsize]
at (axis cs:0.05,1.e-13) {$5.5$};

\end{axis}

\end{tikzpicture}%
\caption{Test case of Section~\ref{sec::acoustic_example}. 
Computed errors $\|e_\varphi\|_{0}$ for the \textit{dG2} (left) and \textit{dG1} (right) method in logarithmic scale as a function of the time step $\Delta t$ for different polynomial degrees $r=1,2,3,4$. We set $N_{el}=400$ polygonal elements and a space polynomial degree equal to $7$.}\label{fig:L2conv_test_dt}
\end{figure}
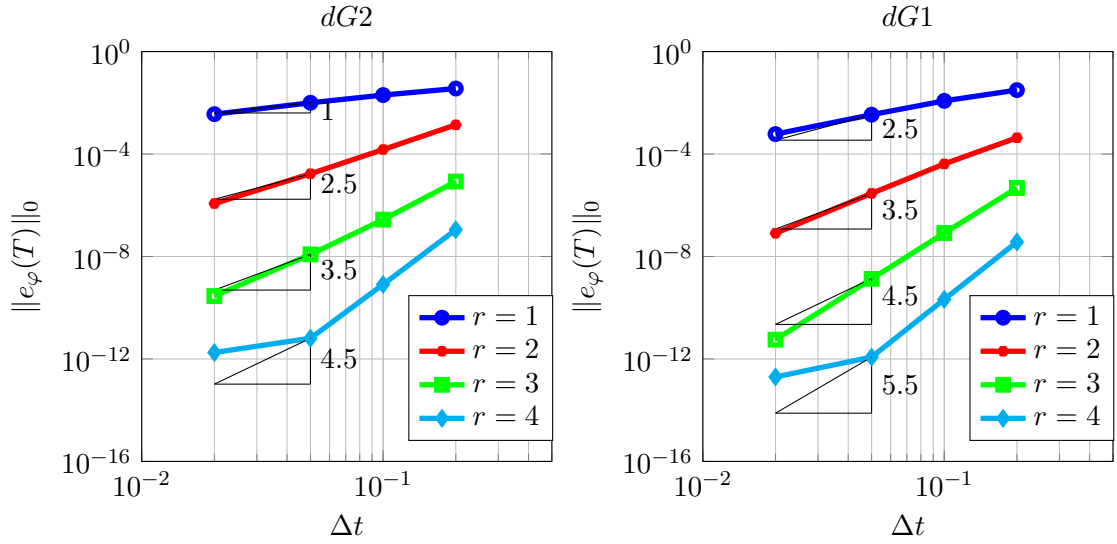
Moreover, in Figure~\ref{fig:acoustic_conv_r}, we present the computed errors of both time-stepping dG schemes versus  the polynomial degree $r$, while maintaining a fixed time step of $\Delta t = 0.02$. Notably, both methods achieve comparable levels of accuracy, with a slightly superior performance observed for the \textit{dG1} scheme.
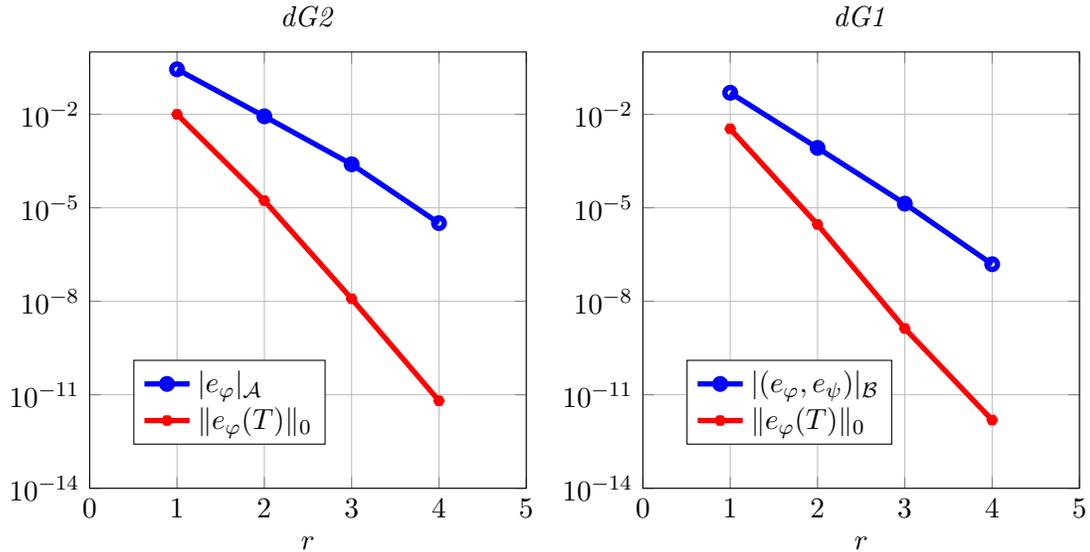
\begin{figure}[!htbp]
\centering
    \begin{tikzpicture}

\begin{axis}[%
width=0.33\textwidth,
height=0.33\textwidth,
scale only axis,
scale only axis,
xmin=0,
xmax=5,
xlabel style={font=\color{black}},
xlabel={$r$},
ymode=log,
ymin=1.e-14,
ymax=1,
yminorticks=true,
ylabel style={font=\color{black}},
title={\textit{dG2}},
axis background/.style={fill=white},
xmajorgrids,
ymajorgrids,
yminorgrids,
legend style={at={(0.1,0.1)}, anchor=south west, legend cell align=left, align=left, draw=black}
]
\addplot [color=blue, line width=2.0pt, mark=o, mark options={solid, blue}]
  table[row sep=crcr]{%
1	2.782793473302495e-01\\
2	8.607811509773690e-03\\
3	2.490685147155186e-04\\
4	3.217088020075614e-06\\
};
\addlegendentry{$|e_\varphi|_{\calA} $}

\addplot [color=red, line width=2.0pt, mark=asterisk, mark options={solid, red}]
  table[row sep=crcr]{%
1	1.001685286449612e-02\\
2	1.695570687628642e-05\\
3	1.207318415062316e-08\\
4	6.430321140445857e-12\\
};
\addlegendentry{$\|e_\varphi(T)\|_{0}$}

\end{axis}

\end{tikzpicture}%
       \begin{tikzpicture}

\begin{axis}[%
width=0.33\textwidth,
height=0.33\textwidth,
scale only axis,
scale only axis,
xmin=0,
xmax=5,
xlabel style={font=\color{black}},
xlabel={$r$},
ymode=log,
ymin=1.e-14,
ymax=1,
yminorticks=true,
ylabel style={font=\color{black}},
title={\textit{dG1}},
axis background/.style={fill=white},
xmajorgrids,
ymajorgrids,
yminorgrids,
legend style={at={(0.1,0.1)}, anchor=south west, legend cell align=left, align=left, draw=black}
]
\addplot [color=blue, line width=2.0pt, mark=o, mark options={solid, blue}]
  table[row sep=crcr]{%
1	4.917641987742438e-02\\
2	8.370799345739045e-04\\
3	1.356810995003859e-05\\
4	1.550458527983530e-07\\
};
\addlegendentry{$|(e_\varphi,e_{\psi})|_{\calB} $}

\addplot [color=red, line width=2.0pt, mark=asterisk, mark options={solid, red}]
  table[row sep=crcr]{%
1	3.430657539125001e-03\\
2	2.913434535096544e-06\\
3	1.341952400937620e-09\\
4	1.543110124039874e-12\\
};
\addlegendentry{$\|e_\varphi(T)\|_{0}$}

\end{axis}

\end{tikzpicture}%
\caption{Test case of Section~\ref{sec::acoustic_example}. 
Computed errors for the \textit{dG2} (left) and \textit{dG1} (right) method in semilogarithmic scale as a function of the polynomial degree $r$ for $\Delta t = 0.02$. We use $N_{el}=400$ polygonal elements and a space polynomial degree equal to 
$7$.}\label{fig:acoustic_conv_r}
\end{figure}

\begin{figure}[!htbp]
\centering
    \input{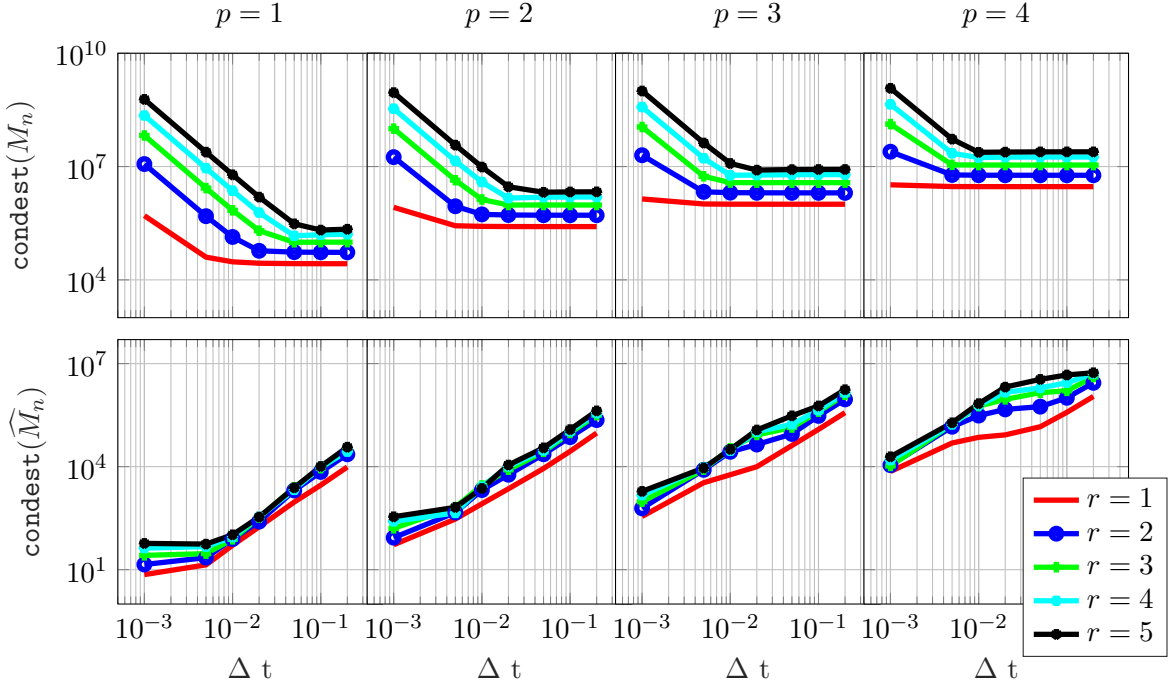} 
\caption{Test case of Section~\ref{sec::acoustic_example}. 
Computed condition number for the \textit{dG2} (top row) and \textit{dG1} (bottom row) discretization in semi-logarithmic scale as a function of the time step $\Delta t$ for different time polynomial degree $r$ and different space polynomial degree $p$. 
}\label{fig:acoustic_cond}
\end{figure}
Finally, we compare the condition number of the system matrix $M_n$ in \eqref{Eq:LinearSystem_I} (resp. $\widehat{M}_n$ in \eqref{eq::linear_system_dGII}) for the \textit{dG2} (resp. \textit{dG1}). 
In Figure~\ref{fig:acoustic_cond} we report the estimated condition numbers (Matlab function \texttt{condest}) by fixing $400$ mesh elements and varying the space polynomial degree $p$, the time step $\Delta t$ and the time polynomial degree $r$. 
For a fixed space polynomial degree 
$p$, we observe a contrasting trend in the condition numbers of the matrices as the time step decreases. While on the one hand, the condition number of $M_n$ increases, on the other hand, the {condition number}  of $\widehat{M}_n$ decreases. 
Specifically, we note that the {condition number} of matrix $\widehat{M}_n$ consistently remains lower than that of $M_n$, especially noticeable for smaller valued of $\Delta t$. Additionally, we observe a correlation between the increase in {condition number} and the polynomial degree $r$.
\textcolor{black}{The results suggest that \textit{dG1} method outperforms} the \textit{dG2} one, at least for the considered test case. Consequently, in the next sections, we will exclusively consider the \textit{dG1} method.

\subsection{Wave propagation in a poroelastic medium}\label{sec::poro_example}
In a poroelastic domain $\Omega$, we consider the low-frequency Biot's equations \cite{biot1}:
\begin{equation}\label{eq::poroelasticity}
\begin{cases}
\rho_p\ddot{\bm{u}}_p + 
\rho_f\ddot{\bm{u}}_f
+ 2\rho_p \zeta \dot{\bm{u}}_p  + \rho_p\zeta^2 \bm{u}_p
-\nabla\cdot\bm{\sigma}_p(\bm{u}_p,\bm{u}_f)=\bm{f}_p,  &\text{in }\Omega\times(0,T],\\
\rho_f\ddot{\bm{u}}_p + 
\rho_w\ddot{\bm{u}}_f + 
\frac{\eta}{k}\dot{\bm{u}}_f
+\nabla p_p(\bm{u}_p,\bm{u}_f)=\bm{f}_f,  &\text{in }\Omega\times(0,T],\\
\bm{\sigma}_p(\bm{u}_p,\bm{u}_f) = \bm{\sigma}_e(\bm{u}_p) 
-\beta\, p_p(\bm{u}_p,\bm{u}_f) \bm{I}, &\text{in }\Omega\times(0,T],\\
 p_p(\bm{u}_p,\bm{u}_f) = -m(\beta \nabla\cdot\bm{u}_p+\nabla\cdot \bm{u}_f),
&\text{in }\Omega\times(0,T],\\
(\bm u_p,\bm u_f)  = (\bm g_p, \bm g_f), & \text{on }\Gamma\times(0,T],\\
  (\bm u_p, \dot{\bm{u}}_p)(0) = (\bm u_{p0}, \bm v_{p0}), & \text{in }\Omega, \\
  (\bm u_f, \dot{\bm{u}}_f)(0) = (\bm u_{f0}, \bm v_{f0}), & \text{in }\Omega.
\end{cases}
\end{equation}
In system \eqref{eq::poroelasticity}, $\bm{u}_p$ and $\bm{u}_f$ represent the solid and filtration displacements, respectively, $\rho_p$ is the average density given by $\rho_p=\phi\rho_f+(1-\phi)\rho_s$, where $\rho_s>0$ is the solid density, $\rho_f>0$ is the saturating fluid density,  $\rho_w$ is defined as $\rho_w=\frac{a}{\phi}\rho_f$, being $\phi$ the porosity satisfying $0<\phi_0\leq\phi\leq\phi_1<1$, and being $a>1$ the tortuosity measuring the deviation of the fluid paths from straight streamlines. The dynamic viscosity of the fluid is given by $\eta>0$, the absolute permeability by $k>0$, while the Biot--Willis's coefficient $\beta$ and the  Biot's modulus  $m$  are such that $\phi<\beta\le1$ and $m\ge m_0>0$. In  \eqref{eq::poroelasticity}, $\bm{f}_p,\bm{f}_f,\bm g_p,\bm g_f, \bm u_{p0}, \bm v_{p0}, \bm u_{f0}$, and $\bm v_{f0}$ are given (regular enough) data.\\

By applying a PolydG space discretization as in \cite{AntoniettiMazzieriNatipoltri2021}, we obtain the system \eqref{Eq:SecondOrderEquation} in which $M,D$ are positive definite and $A$ is positive semidefinite, cf. case b) of Assumption~\ref{ass:matrices}. 
We consider $\Omega=(-1,0) \times (-1,1)$ partitioned into $300$ polygonal elements and fix the space polynomial degree equal to $6$. We consider a problem having the following exact solution
\begin{align*}
 \bm u_{p,ex}(x,t) & = \cos(\sqrt{2}\pi t)(x^2\cos(\pi x/2)\sin(\pi x),x^2\cos(\pi x/2)\sin(\pi x))^T, \\
 \bm u_{f,ex}(x,t) & = -\bm u_{p,ex}(x,t), \\
 \bm v_{p,ex}(x,t) & = \dot{\bm u}_{p,ex}(x,t), \\
 \bm v_{f,ex}(x,t) & = \dot{\bm u}_{f,ex}(x,t), 
\end{align*}
and compute external forces, boundary data, and initial conditions accordingly. 
The material parameters for this example are listed in Table  \ref{tab:mate_param_poro}.
\begin{table}
    \centering
    \begin{tabular}{c|c|c|c|c|c|c|c|c|c|c}
 \textbf{Field} & $\rho_f$ & $\rho$   &  $\lambda$ & $\mu$ & $a$ & $\phi$ & $\eta$ & $\rho_w$ & $\beta$ & $m$\\
 \hline
  \textbf{Value} &  1 &  1& 1& 1& 1& 0.5 & 1 & 2 & 1 & 1
    \end{tabular}
    \caption{Test case of Section~\ref{sec::poro_example}: physical parameters.
    }
    \label{tab:mate_param_poro}
\end{table}

In Figure~\ref{fig:Energy_poro_conv_test_dt} we plot the {computed} energy errors $|(\bm e_{ u}, \bm e_{v})|_{\calB} = |(\bm u_{ex},\bm v_{ex}) - (\bm u_{dG},\bm v_{dG})|_{\calB}$ for the  \textit{dG1} method as a function of the time step $\Delta t$, for polynomial degrees of time $r=1,..,4$ and final time $T=0.6$. 
{The results reported in Figure~\ref{fig:Energy_poro_conv_test_dt} have been obtained with the \textit{dG1} formulation described in Section~\ref{sec::dg_form_I}.}
It is possible to {observe} that the results agree with the theoretical findings in Section~\ref{sec:Convergence}, cf. Theorem ~\ref{thm::conv_dgII}.
\begin{figure}[!htbp]
\centering
%
%
\begin{tikzpicture}

\begin{axis}[%
width=0.33\textwidth,
height=0.33\textwidth,
scale only axis,
xmode=log,
xmin=0.01,
xmax=0.5,
xminorticks=true,
xlabel style={font=\color{black}},
xlabel={$\Delta t $},
ymode=log,
ymin=1e-13,
ymax=1,
yminorticks=true,
ylabel style={font=\color{black}},
ylabel={$|(\bm e_{u}, \bm e_{v})|_{\calB}$},
axis background/.style={fill=white},
title ={$dG1$},
xmajorgrids,
xminorgrids,
ymajorgrids,
yminorgrids,
legend style={at={(0.65,0.05)}, anchor=south west, legend cell align=left, align=left, draw=black}
]
\addplot [color=blue, line width=2.0pt, mark=o, mark options={solid, blue}]
  table[row sep=crcr]{%
     2.000000000000000e-01     3.582368442756189e-01\\
     1.000000000000000e-01     1.822532812825217e-01\\
     5.000000000000000e-02     8.418101062298784e-02\\
     2.000000000000000e-02     3.009151849749764e-02\\
};
\addlegendentry{$r=1$}

\addplot [color=red, line width=2.0pt, mark=asterisk, mark options={solid, red}]
  table[row sep=crcr]{%
     2.000000000000000e-01     3.766972051966913e-02\\
     1.000000000000000e-01     9.593043635578887e-03\\
     5.000000000000000e-02     2.213110206130020e-03\\
     2.000000000000000e-02     3.106172609821840e-04\\
};
\addlegendentry{$r=2$}

\addplot [color=green, line width=2.0pt, mark=square, mark options={solid, green}]
  table[row sep=crcr]{%
     2.000000000000000e-01     2.371968668634173e-03\\
     1.000000000000000e-01     2.784217353134650e-04\\
     5.000000000000000e-02     3.367611648059894e-05\\
     2.000000000000000e-02     2.036425937856025e-06\\
};
\addlegendentry{$r=3$}

\addplot [color=cyan, line width=2.0pt, mark=diamond, mark options={solid, cyan}]
  table[row sep=crcr]{%
     2.000000000000000e-01     1.311732403275686e-04\\
     1.000000000000000e-01     8.354456062754148e-06\\
     5.000000000000000e-02     5.273759745327901e-07\\
     2.000000000000000e-02     1.931122255433458e-07 \\
};
\addlegendentry{$r=4$}

\addplot [color=black, forget plot]
  table[row sep=crcr]{%
0.05	8.418101062298784e-02\\
0.02	2.129629834427772e-02\\
};

\addplot [color=black, forget plot]
  table[row sep=crcr]{%
0.05	2.129629834427772e-02\\
0.02	2.129629834427772e-02\\
};
\addplot [color=black, forget plot]
  table[row sep=crcr]{%
0.05	8.418101062298784e-02\\
0.05	2.129629834427772e-02\\
};

\addplot [color=black, forget plot]
  table[row sep=crcr]{%
0.05	2.213110206130020e-03 \\
0.02	2.239510068587391e-04\\
};
\addplot [color=black, forget plot]
  table[row sep=crcr]{%
0.05	2.239510068587391e-04\\
0.02	2.239510068587391e-04\\
};
\addplot [color=black, forget plot]
  table[row sep=crcr]{%
0.05	2.239510068587391e-04 \\
0.05	2.213110206130020e-03\\
};

\addplot [color=black, forget plot]
  table[row sep=crcr]{%
0.05	3.367611648059894e-05 \\
0.02	1.363113354596175e-06\\
};
\addplot [color=black, forget plot]
  table[row sep=crcr]{%
0.05	1.363113354596175e-06\\
0.02	1.363113354596175e-06\\
};
\addplot [color=black, forget plot]
  table[row sep=crcr]{%
0.05	1.363113354596175e-06\\
0.05	3.367611648059894e-05 \\
};

\addplot [color=black, forget plot]
  table[row sep=crcr]{%
0.05	5.273759745327901e-07\\
0.02	8.538671425405796e-09\\
};
\addplot [color=black, forget plot]
  table[row sep=crcr]{%
0.05	8.538671425405796e-09\\
0.02	8.538671425405796e-09\\
};
\addplot [color=black, forget plot]
  table[row sep=crcr]{%
0.05	8.538671425405796e-09\\
0.05	5.273759745327901e-07\\
};

\node[right, align=left, text=black, font=\normalsize]
at (axis cs:0.05,0.04) {$1.5$};

\node[right, align=left, text=black, font=\normalsize]
at (axis cs:0.05,0.0006) {$2.5$};

\node[right, align=left, text=black, font=\normalsize]
at (axis cs:0.05,0.000006) {$3.5$};

\node[right, align=left, text=black, font=\normalsize]
at (axis cs:0.05,0.00000007) {$4.5$};

\end{axis}

\end{tikzpicture}%
%
%
\begin{tikzpicture}

\begin{axis}[%
width=0.33\textwidth,
height=0.33\textwidth,
scale only axis,
xmode=log,
xmin=0.01,
xmax=0.5,
xminorticks=true,
xlabel style={font=\color{black}},
xlabel={$\Delta t $},
ymode=log,
ymin=1e-13,
ymax=1,
yminorticks=true,
ylabel style={font=\color{black}},
ylabel={$\|  \bm e_{u}(T) \|_0 $},
axis background/.style={fill=white},
title ={$dG1$},
xmajorgrids,
xminorgrids,
ymajorgrids,
yminorgrids,
legend style={at={(0.65,0.05)}, anchor=south west, legend cell align=left, align=left, draw=black}
]
\addplot [color=blue, line width=2.0pt, mark=o, mark options={solid, blue}]
  table[row sep=crcr]{%
     2.000000000000000e-01     3.012172992468890e-02\\
     1.000000000000000e-01     9.578014669782664e-03\\
     5.000000000000000e-02     2.524502166449470e-03\\
     2.000000000000000e-02     3.918299428200966e-04\\
};
\addlegendentry{$r=1$}

\addplot [color=red, line width=2.0pt, mark=asterisk, mark options={solid, red}]
  table[row sep=crcr]{%
     2.000000000000000e-01     8.044944783835343e-04\\
     1.000000000000000e-01     5.678615712352405e-05\\
     5.000000000000000e-02     3.812874191386652e-06\\
     2.000000000000000e-02     1.133258797351962e-07\\
};
\addlegendentry{$r=2$}

\addplot [color=green, line width=2.0pt, mark=square, mark options={solid, green}]
  table[row sep=crcr]{%
     2.000000000000000e-01     2.147508067834767e-05\\
     1.000000000000000e-01     9.765976687311475e-07\\
     5.000000000000000e-02     4.700002506462467e-08\\
     2.000000000000000e-02     7.142513187908134e-10\\
};
\addlegendentry{$r=3$}

\addplot [color=cyan, line width=2.0pt, mark=diamond, mark options={solid, cyan}]
  table[row sep=crcr]{%
     2.000000000000000e-01     9.754264937192897e-07 \\
     1.000000000000000e-01     2.358522580459352e-08 \\
     5.000000000000000e-02     5.791214771070619e-10 \\
     2.000000000000000e-02     1.416316385045191e-10 \\
};
\addlegendentry{$r=4$}

\addplot [color=black, forget plot]
  table[row sep=crcr]{%
0.05	2.524502166449470e-03\\
0.02	2.554616577283146e-04\\
};

\addplot [color=black, forget plot]
  table[row sep=crcr]{%
0.05	2.554616577283146e-04\\
0.02	2.554616577283146e-04\\
};
\addplot [color=black, forget plot]
  table[row sep=crcr]{%
0.05	2.524502166449470e-03\\
0.05	2.554616577283146e-04\\
};

\addplot [color=black, forget plot]
  table[row sep=crcr]{%
0.05	3.812874191386652e-06  \\
0.02	1.543342960186186e-07\\
};
\addplot [color=black, forget plot]
  table[row sep=crcr]{%
0.05	1.543342960186186e-07\\
0.02	1.543342960186186e-07\\
};
\addplot [color=black, forget plot]
  table[row sep=crcr]{%
0.05	3.812874191386652e-06  \\
0.05	1.543342960186186e-07\\
};

\addplot [color=black, forget plot]
  table[row sep=crcr]{%
0.05	 4.700002506462467e-08\\
0.02	7.609709019607881e-10\\
};
\addplot [color=black, forget plot]
  table[row sep=crcr]{%
0.05	7.609709019607881e-10\\
0.02	7.609709019607881e-10\\
};
\addplot [color=black, forget plot]
  table[row sep=crcr]{%
0.05	 4.700002506462467e-08\\
0.05	7.609709019607881e-10\\
};

\addplot [color=black, forget plot]
  table[row sep=crcr]{%
0.05	5.791214771070619e-10\\
0.02	3.750590278818559e-12\\
};
\addplot [color=black, forget plot]
  table[row sep=crcr]{%
0.05	3.750590278818559e-12\\
0.02	3.750590278818559e-12\\
};
\addplot [color=black, forget plot]
  table[row sep=crcr]{%
0.05	5.791214771070619e-10\\
0.05	3.750590278818559e-12\\
};

\node[right, align=left, text=black, font=\normalsize]
at (axis cs:0.05,0.001) {$2.5$};

\node[right, align=left, text=black, font=\normalsize]
at (axis cs:0.05,0.000001) {$3.5$};

\node[right, align=left, text=black, font=\normalsize]
at (axis cs:0.05,0.00000001) {$4.5$};

\node[right, align=left, text=black, font=\normalsize]
at (axis cs:0.05,5.e-11) {$5.5$};

\end{axis}

\end{tikzpicture}%
\caption{Test case of Section~\ref{sec::poro_example}. 
Computed errors $|(\bm e_{u}, \bm e_{ v})|_{\calB}$ (left) and $\|\bm e_{u}\|_{0}$ (right) for the \textit{dG1} as a function of the time step $\Delta t$ for different polynomial degrees $r=1,2,3,4$ {(semi-log scale)}. {The results have been obtained on a space mesh with} $N_{el}=300$ polygonal elements and a space polynomial degree equal to $6$.}
\label{fig:Energy_poro_conv_test_dt}
\end{figure}
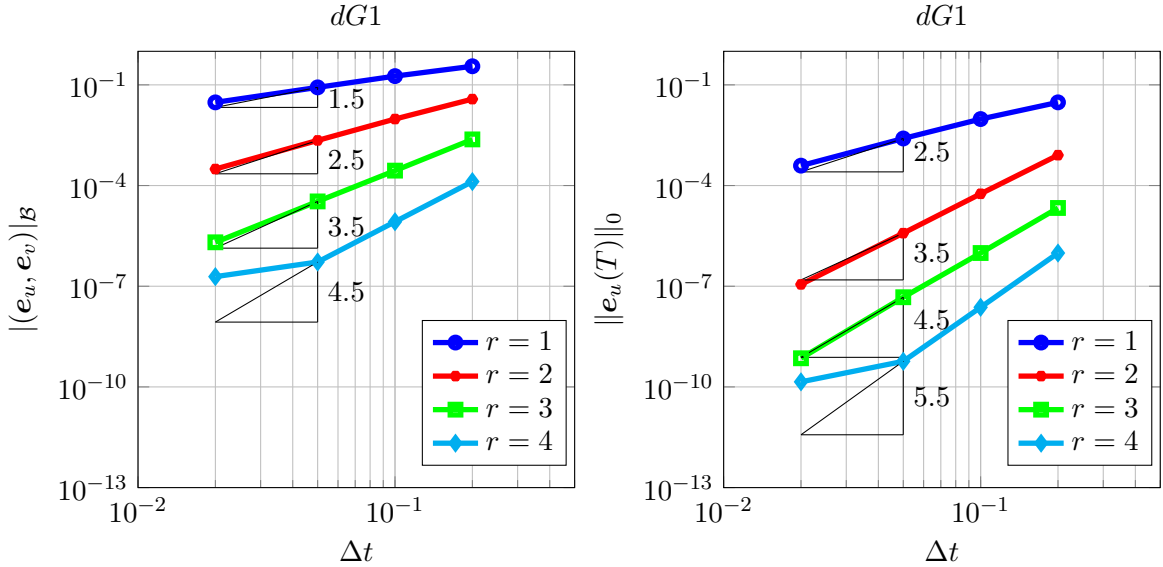
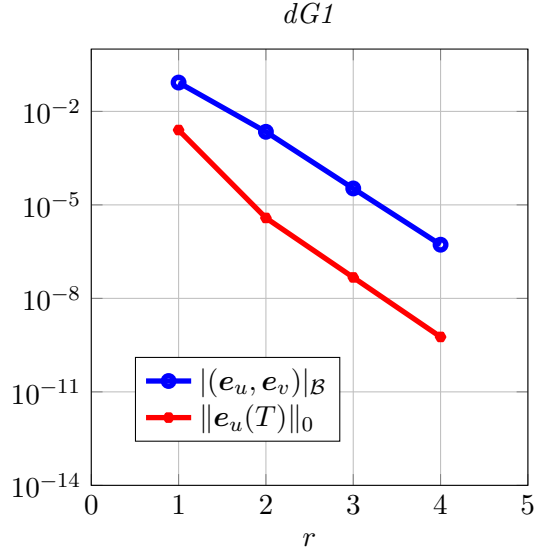
\begin{figure}[!htbp]
\centering
    \begin{tikzpicture}

\begin{axis}[%
width=0.33\textwidth,
height=0.33\textwidth,
scale only axis,
scale only axis,
xmin=0,
xmax=5,
xlabel style={font=\color{black}},
xlabel={$r$},
ymode=log,
ymin=1.e-14,
ymax=1,
yminorticks=true,
ylabel style={font=\color{black}},
title={\textit{dG1}},
axis background/.style={fill=white},
xmajorgrids,
ymajorgrids,
yminorgrids,
legend style={at={(0.1,0.1)}, anchor=south west, legend cell align=left, align=left, draw=black}
]
\addplot [color=blue, line width=2.0pt, mark=o, mark options={solid, blue}]
  table[row sep=crcr]{%
1	8.418101062298784e-02\\
2	2.213110206130020e-03\\
3	3.367611648059894e-05\\
4   5.273759745327901e-07 \\
};
\addlegendentry{$|(\bm e_{u}, \bm e_{ v})|_{\calB} $}

\addplot [color=red, line width=2.0pt, mark=asterisk, mark options={solid, red}]
  table[row sep=crcr]{%
1	2.524502166449470e-03\\
2	3.812874191386652e-06\\
3	4.700002506462467e-08\\
4 5.791214771070619e-10 \\
};
\addlegendentry{$\|\bm e_{u}(T)\|_{0}$}

\end{axis}

\end{tikzpicture}%
\caption{Test case of Section~\ref{sec::poro_example}. 
{Computed errors for the \textit{dG1}  method  as a function of the polynomial degree $r$ for $\Delta t = 0.05$ (semi-log scale)}. {The results have been obtained on a space mesh with} $N_{el}=300$ polygonal elements and a space polynomial degree equal to 
$6$.}\label{fig:poro_conv_r}
\end{figure}

\subsection{Wave propagation in a coupled poroelastic-acoustic domain}\label{sec:coupled-poro-acoustic} 
In our final applicative example, we consider a domain $\Omega$ composed of a poroelastic medium $\Omega_p$ and an acoustic medium $\Omega_a$, defined as $\Omega = \Omega_p \cup \Omega_a$. These domains share an interface boundary denoted as $\Gamma_I$.
The circular porous cylinder $\Omega_p$, is surrounded by the acoustic medium $\Omega_a$. We consider a circular interface $\Gamma_I$ of radius $100~m$ centered at $(0, 0)$ in the domain $\Omega = (-600,600)~m^2$, see Figure~\ref{fig:AcousticCavity-domain}, cf. \cite{chiavassa_lombard_2013}.
In $\Omega_a$ (resp. $\Omega_p$) we consider
\eqref{eq::acoustic} (resp. \eqref{eq::poroelasticity}) while on $\Gamma_I$
the following coupling conditions are imposed  
\begin{equation}\begin{cases} \label{eq:interface}
-\bm{\sigma}_p(\bm u_p,\bm u_f) \bm{n}_p  = \rho_a\dot{\varphi}_a\bm{n}_p & \textrm{ on }\Gamma_{I}  \times (0,T],    \\
p_p(\bm u_p , \bm u_f ) = \rho_a \dot{\varphi}_a
 & \textrm{ on }\Gamma_{I}  \times (0,T],  \\
-(\dot{\bm{u}_f}+\dot{\bm{u}}_f)\cdot\bm{n}_p = \nabla\varphi_a\cdot\bm{n}_p& \textrm{ on }\Gamma_{I}  \times (0,T],
\end{cases}\end{equation}
expressing the continuity of normal stresses, continuity of pressure, and conservation of mass, respectively. We refer the reader to \cite{ABM_Vietnam} for a detailed problem description. As explained in \cite{AntoniettiMazzieriNatipoltri2021}, space discretization of problem \eqref{eq::acoustic}-\eqref{eq::poroelasticity}-\eqref{eq:interface} with a PolydG method leads to a system of ordinary differential equations of the form \eqref{Eq:SecondOrderEquation} which we integrate in time using the \textit{dG1} method in Section~\ref{sec::dg_form_II}. 
\begin{figure}[!htbp]
    \centering
    \includegraphics[width=0.52\textwidth]{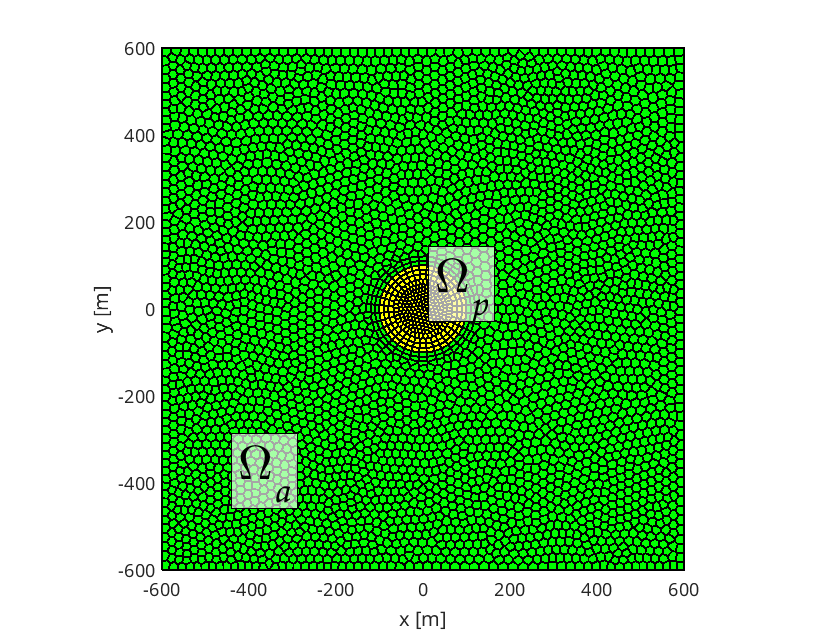}
    \includegraphics[width=0.47\textwidth]{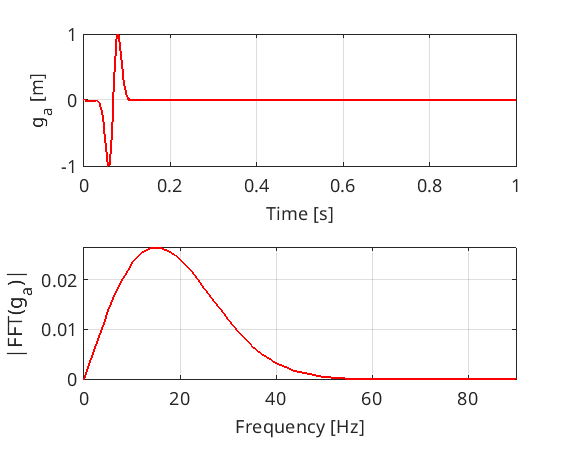}
    \caption{Test case of Section~\ref{sec:coupled-poro-acoustic}. Left: circular porous domain $\Omega_p$ (yellow) surrounded by an acoustic medium $\Omega_a$ (green). The mesh is made by $4287$ polygonal elements divided into $3979$ for $\Omega_a$ ($h\approx 36.5~m$) and $308$ for $\Omega_p$ ($h\approx 11.8~m$). Right: Dirichlet boundary condition applied on the domain's bottom edge (top), and the absolute value of its Fourier transform (bottom).}
    \label{fig:AcousticCavity-domain}
\end{figure}

The computational domain $\Omega$ is discretized by 4287
polygonal elements having a characteristic size $h\approx 36~m$ and a space polynomial degree $p=4$ is set. A locally refined mesh is employed around the interface to accurately compute the  wave conversions, cf. Figure~\ref{fig:AcousticCavity-domain}.
The source is an acoustic plane wave given as a Dirichlet condition on the bottom boundary $(-600,600)\times \{-600\}$, cf. Figure ~\ref{fig:AcousticCavity-domain} (top-right).
The pulse is a Gaussian wavelet 
\begin{equation}\label{eq:gaussian_pulse}
g_a(t) = 2\pi f_p \sqrt{e}(t - 1/f_p)e^{-2(\pi f_p)^2 (t - 1/f_p)^2},
\end{equation}
having a peak frequency $f_p = 15~Hz$, \textcolor{black}{cf. Figure~\ref{fig:AcousticCavity-domain} (bottom-right)}. 
A sound soft condition is enforced on the remaining boundaries, i.e., $\nabla \varphi_a \cdot \bm n = 0$. The physical parameters for the test case are listed in Table \ref{tab::table_poroacoustic}

In Figure~\ref{fig:snapshot_acousticcavity_15Hz} we plot the computed pressure field ($p_a = \rho_a \dot{\varphi}_a$ in $\Omega_a$ and $p_p = -m(\beta \nabla\cdot\bm{u}_p+\nabla\cdot \bm{u}_f)$ in $\Omega_p$) at different time instants, considering a final time $T=1~s$, a time step $\Delta t = 0.001~s$,
and a time polynomial degree equal to $2$, resulting in a system matrix $\widehat{M}_n$  of dimension 234495, cf. \eqref{eq::linear_system_dGII}. 
It is possible to see the wave moves from the bottom to the top of the domain, impacting the porous cylinder, and producing a scattered circular wavefield directed backward. This agrees with the results presented in \cite{chiavassa_lombard_2013}. 
\begin{table}[htbp]
\centering
\begin{tabular}{llllll}
\cline{1-6}
\textbf{Fluid}  & Fluid density      & $\rho_f$    & 1000    & $\rm kg/m^3$    &  \\
                &                    & $\rho_a$    & 1000   & $\rm kg/m^3$    &  \\
                & Wave velocity      & $c$         & 1500   & $\rm m/s$       &  \\
                & Dynamic viscosity  & $\eta$      & $1.05\cdot 10^{-3}$      & $\rm Pa\cdot s$ &  \\ \cline{1-6}
\textbf{Grain}  & Solid density      & $\rho_s$    & 2690   & $\rm kg/m^3$    &  \\
                & Shear modulus      & $\mu$       & 1.86 $\cdot 10^9$      & $\rm Pa$   &  \\ \cline{1-6}
\textbf{Matrix} & Porosity           & $\phi$      & 0.38    &             &  \\
                & Tortuosity         & $a$         & 1.8      &             &  \\
& Permeability  & $k$                & $2.79\cdot 10^{-11}$    & $\rm m^2$       &  \\
& Lam\'e coefficient   & $\lambda$ & 1.20$\cdot 10^8$ & $\rm Pa$        &  \\
& Biot's coefficient & $m$         & 5.34$\cdot 10^9$ & $\rm Pa$        &  \\
& Biot's coefficient & $\beta$     & 0.95             &             &  \\  \cline{1-6}
\end{tabular}
\caption{Test case of Section~\ref{sec:coupled-poro-acoustic}. Physical parameters for poro-elastic media.}
\label{tab::table_poroacoustic}
\end{table}
\begin{figure}[!htbp]
    \centering
    \includegraphics[width=0.35\textwidth]{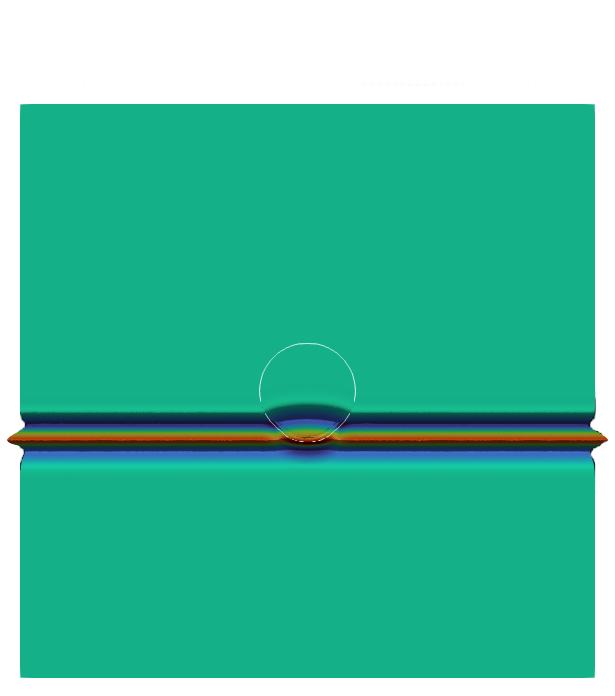}
    \includegraphics[width=0.35\textwidth]{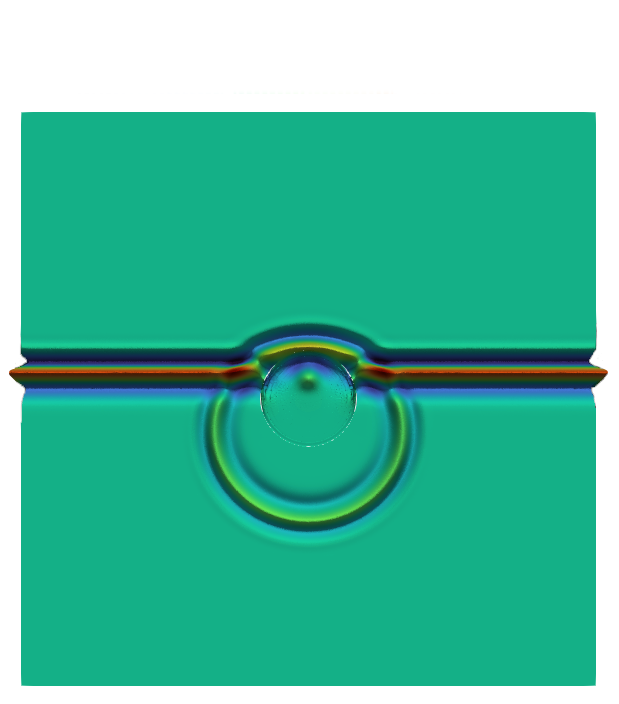}
    \includegraphics[width=0.35\textwidth]{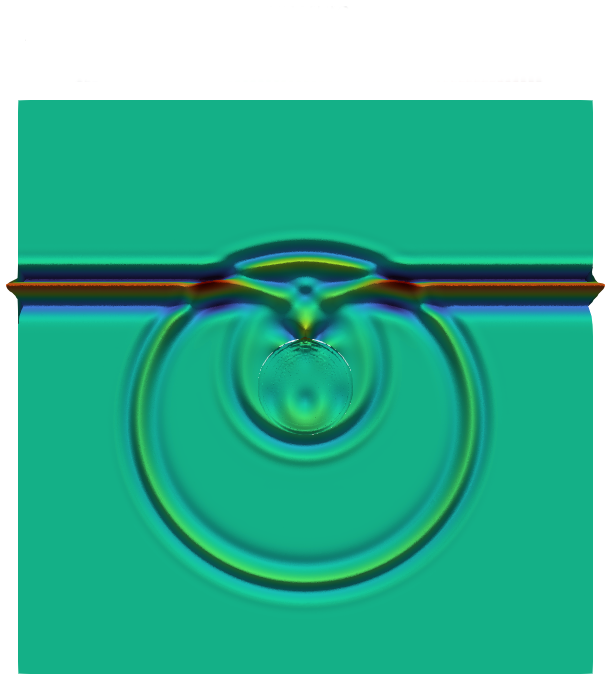}
    \includegraphics[width=0.35\textwidth]{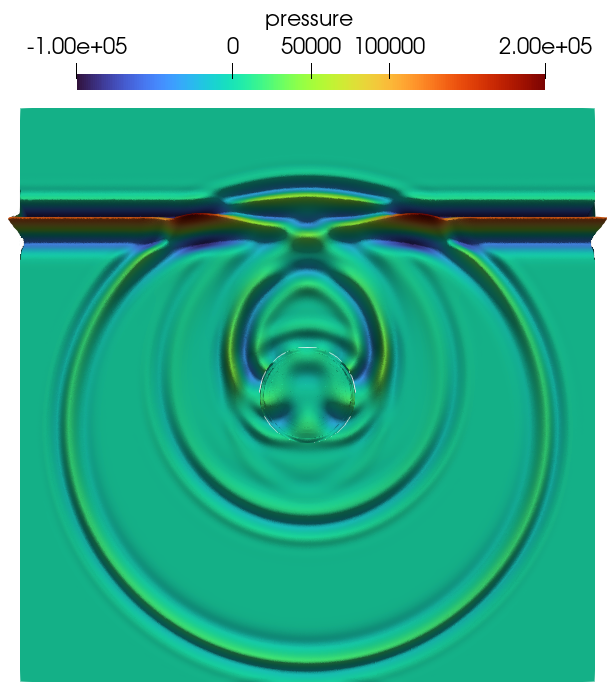}    
    \caption{Test case of Section ~\ref{sec:coupled-poro-acoustic}. Snapshots of the computed pressure field at different time instants: $t=0.4~s$ (top-left), $t=0.5~s$ (top-right), $t=0.6~s$ (bottom-left), and $t=0.7~s$ (bottom-right).}
    \label{fig:snapshot_acousticcavity_15Hz}
\end{figure}

\section{Conclusions}
In this chapter, we reviewed and numerically compared two discontinuous Galerkin time-integration methods for the approximate solution of second-order differential systems stemming from space discretization of wave-type problems. We presented the two formulations within a generalized framework and discussed their stability, accuracy, and computational efficiency. Our numerical study suggested that the formulation proposed for the first-order system outperforms that for the second-order system in terms of both accuracy and computational efficiency. Notably, regarding the latter, it was demonstrated that the matrix of the linear system resulting from the discretization seems to exhibit a better condition number. Additionally, employing an appropriate solution strategy helps to reduce the computational expenses associated with incorporating a new variable into the model. To support our conclusions, we presented a wide set of numerical examples, including wave propagation problems in acoustic, poroelastic, and coupled acoustic-poroelastic media.

\section*{Aknowledgement}
PFA and IM have been partially supported by ICSC--Centro Nazionale di Ricerca in High Performance Computing, Big Data, and Quantum Computing funded by European Union--NextGenerationEU.
The work of GC, and IM has been partially supported by the PRIN2022 grant ASTICE - CUP:
D53D23005710006. The present research is  part of the activities of "Dipartimento di Eccellenza 2023-2027". PFA,  GC, and IM are members of INdAM-GNCS. 
This work is partially funded by the European Union (ERC SyG, NEMESIS, project number 101115663). Views and opinions expressed are however those of the authors only and do not necessarily reflect those of the European Union or the European Research Council Executive Agency.
\textcolor{black}{AA is currently at Leonardo S.p.A., but the research presented here was conducted while AA was at Politecnico di Milano.
}

\end{document}